%% file: paper.tex
\documentclass[final]{dmtcs-episciences}
\usepackage[T1]{fontenc}
\usepackage[round]{natbib}
\usepackage{amsmath}
\usepackage{amssymb}
\usepackage{amsthm}
\usepackage[subpreambles=false]{standalone}
\usepackage{listings}
\usepackage{hyperref}
\usepackage{array}
\usepackage{tikz}
\usetikzlibrary{patterns}
\usetikzlibrary{decorations.pathreplacing}
\input{figures/macro}
\usepackage{xargs}
\usepackage{xcolor}
\usepackage[colorinlistoftodos,prependcaption]{todonotes}
\newcommandx{\emile}[1]{\todo[inline,caption={\'Emile},backgroundcolor=green!25,bordercolor=green,]{#1}}
\newcommandx{\henning}[1]{\todo[inline,caption={Henning},backgroundcolor=blue!25,bordercolor=blue,]{#1}}
\newcommandx{\christian}[1]{\todo[inline,caption={Christian},backgroundcolor=orange!25,bordercolor=orange,]{#1}}

\theoremstyle{definition}
\newtheorem{theorem}{Theorem}[section]
\newtheorem{corollary}[theorem]{Corollary}
\newtheorem{lemma}[theorem]{Lemma}
\newtheorem{proposition}[theorem]{Proposition}

\newtheorem{definition}[theorem]{Definition}
\newtheorem{remark}[theorem]{Remark}

\DeclareMathOperator{\Av}{Av}
\newcommand{\Avp}{\Av^{+}}
\newcommand{\Avn}{\Av^{(n)}}
\DeclareMathOperator{\perms}{\mathcal{S}}
\DeclareMathOperator{\SE}{SE}
\DeclareMathOperator{\WI}{WI}
\DeclareMathOperator{\uperm}{uperm}
\DeclareMathOperator{\dperm}{dperm}
\DeclareMathOperator{\rl}{r}
\newcommand*{\mergecore}{\vee}

\usepackage{mathtools}
\newcommand*{\bstrip}[1]{#1^\times}
\newcommand*{\WIs}{\mathcal{I}}

\DeclareMathOperator{\UU}{U}
\DeclareMathOperator{\D}{D}
\DeclareMathOperator{\C}{C}
\DeclareMathOperator{\R}{R}
\DeclareMathOperator{\UD}{UD}
\DeclareMathOperator{\RC}{RC}
\DeclareMathOperator{\UR}{UR}
\DeclareMathOperator{\DR}{DR}
\DeclareMathOperator{\UDC}{UDC}

\DeclareMathOperator{\UDRC}{UDRC}

\DeclareMathOperator{\gfU}{\mathbf{F}}
\DeclareMathOperator{\gfUD}{\mathbf{W}}
\DeclareMathOperator{\gfUDRC}{\mathbf{Y}}
\DeclareMathOperator{\gfUDC}{\mathbf{G}}
\DeclareMathOperator{\gfDmUR}{\mathbf{H}}
\DeclareMathOperator{\gfUmDR}{\mathbf{J}}

\newcommand{\oeis}[1]{\href{https://oeis.org/#1}{#1}}

\input{patternmacros.tex}
\pgfmathsetmacro{\patttablescale}{1.05}
\pgfmathsetmacro{\pattdispscale}{0.80}
\pgfmathsetmacro{\patttextscale}{0.5}
\newcommand{\mptwodec}{
    $\mpattern{scale=\pattdispscale}{2}{1/2, 2/1}{0/2, 0/1, 1/0, 1/1, 1/2, 2/1, 2/2}$
    }
\newcommand{\mptwoinc}{
    $\mpattern{scale=\pattdispscale}{2}{1/1, 2/2}{0/2, 0/1, 1/0, 1/1, 1/2, 2/1, 2/2}$
    }

\title{Enumeration of Permutation Classes and Weighted Labelled Independent Sets}
\author{Christian Bean\affiliationmark{1}
    \and \'Emile Nadeau\affiliationmark{2}
\and Henning Ulfarsson\affiliationmark{2}}
\affiliation{
    Universit\'e Paris 13, LIPN, France\\
    Reykjavik University, Iceland}

\received{2019-12-20}
\revised{2021-02-16}
\accepted{2021-03-13}
\publicationdetails{22}{2021}{2}{2}{5995}

\begin{document}
\maketitle

\begin{abstract}
    In this paper, we study the staircase encoding of permutations, which maps a
    permutation to a staircase grid with cells filled with permutations.
    We consider many cases, where restricted to	a permutation class,
    the staircase encoding becomes a bijection to its image.
    We describe the image of those
    restrictions using independent sets of graphs weighted with permutations.
    We derive the generating function for the independent sets and then for
    their weighted counterparts.
    The bijections we establish provide the enumeration of
    permutation classes.
    We use our results to
    uncover some unbalanced Wilf-equivalences of permutation classes and outline
    how to do random sampling in the permutation classes.
    In particular, we cover the classes $\mathrm{Av}(2314,3124)$,
    $\mathrm{Av}(2413,3142)$, $\mathrm{Av}(2413,3124)$, $\mathrm{Av}(2413,2134)$
    and $\mathrm{Av}(2314,2143)$, as well as many subclasses.
\end{abstract}

\input{1-intro}
\input{2-encoding}
\input{3-len3to4}

\input{4-up-and-down-cores}
\input{5-updown-core}
\input{6-new-cores}
\input{7-generalization}

\input{8-back-to-2patts}
\input{9-the-crazy-one}
\input{10-the-other-crazy.tex}
\input{11-wilf}
\input{12-conclusion}

\bibliographystyle{abbrvnat}
\bibliography{references.bib}

\end{document}

%% file: figures/macro.tex
\newcommand{\highlightRegion}[2]{
    \begin{scope}
        \clip #2--cycle;
        \draw[#1, line width=1mm, pattern=north east lines, pattern color=#1] #2--cycle;
    \end{scope}
}

\newcommand{\drawGrid}[1]{
    \pgfmathtruncatemacro{\size}{#1}
    \foreach\x in {1,...,\size} { 
        \draw[fill=gray, gray] (\x,\x) circle (3pt);
        \draw[gray] (\x,0)--(\x,\x)--(\size+1,\x);
    }
    \draw[gray] (1,0)--(\size+1,0)--(\size+1,\size);
}

%% file: patternmacros.tex
\newcommand{\shadetheboxesPM}[1]{
    \foreach \x/\y in {#1}
    \fill[pattern color = black!75, pattern=north east lines] (\x,\y) rectangle +(1,1);
}

\newcommand{\drawthegrid}[1]{
    \draw (0.01,0.01) grid (#1+0.99,#1+0.99);
}

\newcommand{\drawverticallines}[3]{
    \foreach \x in {#2}
    \draw[line width=#3] (\x+0.01,0.01) -- (\x+0.01,#1+0.99);
}

\newcommand{\drawhorizontallines}[3]{
    \foreach \y in {#2}
    \draw[line width=#3] (0.01,\y+0.01) -- (#1+0.99,\y+0.01);
}

\newcommand{\drawtheclpattern}[1]{
    \foreach \x/\y in {#1}
    \filldraw (\x,\y) circle (6pt);
}

\newcommand{\drawclpattern}[2]{
	\foreach[count=\x] \y in {#1}
	{
		\filldraw (\x,\y) circle (#2 pt);
	}
}

\newcommand{\drawspecialbox}[1]{
    \foreach \x/\y/\z/\w/\A in {#1}
    {
        \fill[color = white!100, opacity=1, rounded corners = 1.5pt] (\x+0.125,\y+0.125) rectangle (\z-0.125,\w-0.125);
        \draw[color = black, rounded corners = 1.5pt] (\x+0.125,\y+0.125) rectangle (\z-0.125,\w-0.125);
        \fill[black] (\x/2+\z/2,\y/2+\w/2) node {\A};
    }
}

\newcommand{\drawspecialboxlarge}[1]{
    \foreach \x/\y/\z/\w/\A in {#1}
    {
        \fill[color = white!100, opacity=1, rounded corners = 1.5pt] (\x+0.125,\y+0.125) rectangle (\z-0.125,\w-0.125);
        \draw[color = black, rounded corners = 1.5pt] (\x+0.125,\y+0.125) rectangle (\z-0.125,\w-0.125);
        \fill[black] (\x/2+\z/2,\y/2+\w/2) node {\Large \A};
    }
}

\newcommand{\drawsolidshadedbox}[1]{
    \foreach \x/\y/\z/\w/\A in {#1}
    {
        \fill[color = gray!50, opacity=1, rounded corners=1.5pt] (\x+0.125,\y+0.125) rectangle (\z-0.125,\w-0.125);
        \draw[color = black, rounded corners=1.5pt] (\x+0.125,\y+0.125) rectangle (\z-0.125,\w-0.125);
        \fill[black] (\x/2+\z/2,\y/2+\w/2) node {\A};
    }
}

\newcommand{\drawlabels}[1]{
	\foreach \x/\y/\lab in {#1}
	{
		\draw (\x + 0.5,\y + 0.5) node {\lab};
	}
}

\newcommandx{\patt}[9][4={},5={},6={},7={},8={},9=4]
{
	\scalebox{#1}
	{
		\begin{tikzpicture}[baseline=(current bounding box.center)]
			\useasboundingbox (0.0,-.3) rectangle (#2+1,#2+1.3);
			\shadetheboxesPM{#4}
			\draw (0.01,0.01) grid (#2+1-0.01,#2+1-0.01);

			\drawsolidshadedbox{#6}
			\drawspecialbox{#7}
			\drawspecialboxlarge{#5}
			\drawclpattern{#3}{#9}
			\drawlabels{#8}
		\end{tikzpicture}
	}
}

\newcommandx{\cpatt}[8][4={},5={},6={},7={},8={}]
{
	\scalebox{#1}
	{
		\begin{tikzpicture}[baseline=(current bounding box.center)]
			\useasboundingbox (0.0,-.3) rectangle (#2+1,#2+1.3);
			\shadetheboxesPM{#4}
			\draw (0.01,0.01) grid (#2+1-0.01,#2+1-0.01);

			\drawsolidshadedbox{#6}
			\drawspecialbox{#7}
			\drawspecialboxlarge{#5}
			\drawclpattern{#3}{4}

			\foreach \x/\y in {#8}
			{
				\draw[line width=1] (\x,\y) circle (7 pt);
			}
		\end{tikzpicture}
	}
}

\newcommandx{\metapatt}[8][6={},7={},8={}]
{
    \scalebox{#1}
    {
        \begin{tikzpicture}[baseline=(current bounding box.center)]
					\foreach \width/\height in {#2}
					{
						\useasboundingbox (0.0,-.3) rectangle (\width+1,\height+1.3);
            \shadetheboxesPM{#6}

            \foreach \pos/\type in {#4}
            {
                \ifthenelse{\equal{\type}{v}}
                {
                    \drawverticallines{\height}{\pos}{1.7pt}
                }
                {
								    \ifthenelse{\equal{\type}{d}}
                    {
                      \draw[densely dashed] (\pos,0) -- (\pos,\height+1);
                    }
										{
											\drawhorizontallines{\width}{\pos}{1.7pt}
										}
                }
            }

            \foreach \pos/\type in {#3}
            {
                \ifthenelse{\equal{\type}{v}}
                {
                    \drawverticallines{\height}{\pos}{0.6pt}
                }
                {
										\drawhorizontallines{\width}{\pos}{0.6pt}
                }
            }

            \drawsolidshadedbox{#8}
            \drawspecialbox{#7}

            \foreach \x/\y/\type in {#5}
            {
                \ifthenelse{\equal{\type}{a}}
                {
                    \draw (\x,\y) circle (6pt);
                    \filldraw (\x,\y) circle (3pt);
                }
                {
                    \filldraw (\x,\y) circle (4pt);
                }
            }
					}
        \end{tikzpicture}
    }
}

\newcommandx{\dpatt}[9][6={},7={},8={},9={}]
{
    \scalebox{#1}
    {
        \begin{tikzpicture}[baseline=(current bounding box.center)]
					\foreach \width/\height in {#2}
					{
						\useasboundingbox (0.0,-.3) rectangle (\width+1,\height+1.3);
            \shadetheboxesPM{#6}

            \foreach \pos/\type in {#4}
            {
                \ifthenelse{\equal{\type}{v}}
                {
                    \drawverticallines{\height}{\pos}{1.7pt}
                }
                {
								    \ifthenelse{\equal{\type}{d}}
                    {
                      \draw[densely dashed] (\pos,0) -- (\pos,\height+1);
                    }
										{
											\drawhorizontallines{\width}{\pos}{1.7pt}
										}
                }
            }

            \foreach \pos/\type in {#3}
            {
                \ifthenelse{\equal{\type}{v}}
                {
                    \drawverticallines{\height}{\pos}{0.6pt}
                }
                {
										\drawhorizontallines{\width}{\pos}{0.6pt}
                }
            }

            \drawsolidshadedbox{#8}
            \drawspecialbox{#7}

            \foreach \x/\y/\type in {#5}
            {
                \ifthenelse{\equal{\type}{a}}
                {
                    \draw9 (\x,\y) circle (6pt);
                    \filldraw (\x,\y) circle (3pt);
                }
                {
                    \filldraw (\x,\y) circle (4pt);
                }
            }

						\drawlabels{#9}
					}
        \end{tikzpicture}
    }
}

\newcommand{\mpattern}[4]{
  \raisebox{0.6ex}{
  \begin{tikzpicture}[scale=0.35, baseline=(current bounding box.center), #1]
  	\useasboundingbox (0.0,-0.1) rectangle (#2+1.4,#2+1.1);

    \shadetheboxesPM{#4}

    \drawthegrid{#2}

    \drawtheclpattern{#3}

  \end{tikzpicture}}
}

%% file: 1-intro.tex

\section{Introduction}

A \emph{permutation of size n} is  an arrangement of the numbers
$1,2,\ldots, n$. For example, $43251$ is a permutation of size 5. The set of all
permutations is $\perms$, and the subset of permutations of size $n$ is $\perms_n$.
The \emph{standardization} of a
string $s$ of distinct integers is the permutation obtained by replacing the
$i$-th smallest entry by $i$. A permutation $\sigma$ \emph{contains} a
permutation
$\pi$ if $\pi$ is the standardization of a (not necessarily
consecutive) substring $s$ of $\sigma$. In this context, we say that $\pi$ is a
\emph{(classical) pattern} in $\sigma$. The substring $s$ is called an
\emph{occurrence} of
$\pi$ in $\sigma$. If $\sigma$ does not contain the pattern $\pi$, we say that
$\sigma$ \emph{avoids} $\pi$. We represent the permutation $\sigma$ on a
grid by
placing points at coordinates $(i, \sigma(i))$.
For example, the permutation
$\sigma=194532678$ is pictured in Figure~\ref{fig:pattern}.
The pattern $\pi = 4321$ is contained in $\sigma$ since it is the
standardization
of the substring $9532$ of $\sigma$. The permutation $\sigma$ avoids the pattern
$54321$ since it does not contain any decreasing substring
of size $5$.

\begin{figure}[htpb]
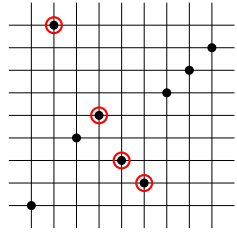

    \centering
    \includestandalone{figures/pattern}
    \caption{The permutation $194532678$ with an occurrence of the
        pattern $4321$ (circled in red).}%
    \label{fig:pattern}
\end{figure}

Given a pattern $\pi$ we define $\Av_n(\pi)$ as the set of permutations of size
$n$ that avoid $\pi$, and $\Av(\pi)$ as $\cup_{n \geq 0} \Av_n(\pi)$.
For a set of
patterns $P$, we let $\Av_n(P) = \cap_{\pi \in P} \Av_n(\pi)$, and
$\Av(P) = \cup_{n \geq 0} \Av_n(P)$. The set of non-empty permutations of $\Av(P)$ is denoted $\Avp(P)$.
A \emph{permutation class} $\mathcal{C}$ is a set
of permutations that is \emph{closed downwards}, in the sense that if $\sigma \in \mathcal{C}$
then $\pi \in \mathcal{C}$ for any pattern $\pi$ in $\sigma$. It turns out that
for a a permutation class $\mathcal{C}$ we can always find a set of patterns $P$ such
that $\mathcal{C} = \Av(P)$, for instance we can take $P = \perms \setminus \mathcal{C}$.
However, the the smallest set of patterns $B$ such that
$\mathcal{C} = \Av(B)$ is called the \emph{basis} of $\mathcal{C}$. There are cases where
the basis consists of infinitely many patterns, due to the existence of infinite anti-chains
in the containment order on $\perms$, but in this paper we only consider finite bases.

We define the \emph{sum} of two permutations $\alpha$ and $\beta$ as the
permutation
\[
    \alpha\oplus\beta=\alpha_1\cdots\alpha_n(\beta_1+n)\cdots(\beta_m+n).
\]
Similarly, the \emph{skew-sum}  of $\alpha$ and $\beta$ is
\[
    \alpha\ominus\beta=(\alpha_1+m)\cdots(\alpha_n+m)\beta_1\cdots\beta_m.
\]
In Figure~\ref{fig:sums}, we see the sum and skew-sum of $123$ and $21$.
If a permutation $\sigma$ can be expressed as the sum of two non-empty
permutations we say that $\sigma$ is \emph{sum-decomposable}.
The permutation $12435$ is sum-decomposable since it can be expressed as
$12\oplus213$.
Similarly, we say that a permutation is \emph{skew-decomposable} if it can
be written as the skew-sum of two non-empty permutations. The permutation
$43512$ is skew-decomposable since it is $213\ominus 12$.
A permutation is \emph{sum-indecomposable} if is not sum-decomposable and
\emph{skew-indecomposable} if it is not skew-decomposable.
The sum of a permutation $\sigma$ and a set of
permutations $P$ is the set $\sigma\oplus P = \{ \sigma\oplus\alpha : \alpha\in P\}$,
and likewise $\sigma\ominus P = \{ \sigma\ominus\alpha : \alpha\in P\}$. We
also define the sum, and skew-sum, of two sets of permutations in the obvious way.

\begin{figure}[htpb]
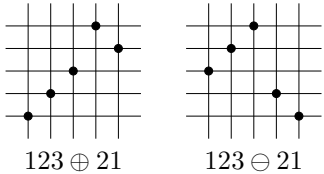

    \centering
    \includestandalone{figures/sums}
    \caption{The sum and skew-sum of two permutations.}%
    \label{fig:sums}
\end{figure}

\cite{bean2015pattern}
introduced the \emph{staircase encoding},
a function which maps a permutation to a staircase grid where cells are
filled with non-negative integers.
In this context, each integer is the size of the monotone sequence in its cell.
In this paper, we refine the staircase encoding as
a function which maps a permutation to a
staircase grid where cells are filled with permutations.
Using this function, we
retrieve the \emph{generating function} $\sum^{}_{n>0} \left|\Av_n(P)\right|
    x^n$ of several permutation classes.

We first recover the results for $\Av(123)$ and $\Av(132)$ from
\cite{bean2015pattern} in Section~\ref{sec:perm_on_grid}, using our more
refined encoding and weighted independent sets.
Our technique is then extended to describe the structure of $\Av(2314, 3124)$
and $\Av(2413, 3142)$ in Section~\ref{sec:3to4} and~\ref{sec:inf-up-down}.
In Section~\ref{sec:UDRC-core}, we recall the \emph{updown core graph}
introduced by \cite{bean2015pattern}
and use it to enumerate $\Av(2314, 3124, 2413, 3142)$, before introducing a new
core graph in Section~\ref{sec:new-core} that is used to give the
structure of $\Av(2314, 3124, 3142)$. Our notion of weighted independent sets
is then generalized to allow labelling.
This enables for a more refined choice
of permutations in our encoding, and is used to enumerate
$\Av(2413, 3142, 3124)$ in Section~\ref{sec:gen_inflation} and
$\Av(2413, 3124)$ in Section~\ref{sec:ru_cd}.
By allowing some
interleaving between cells in the staircase grid representation of a
permutation, we obtain the enumeration for $\Av(2413, 2134)$ and
$\Av(2314, 2134)$ in Sections~\ref{sec:rd_2134} and~\ref{sec:ru_2143}.
Finally, in Section~\ref{sec:wilf} we use
results from previous sections to prove two unbalanced Wilf-equivalences.
Our results handle in a
unified framework the enumeration of many permutation classes that were first
enumerated
in \cite{MR3211768,MR1948771,MR3633253,MR1754331,MR1997910,MR2156679,MR3659386}.
Moreover, the results also allow one to easily enumerate many subclasses of these classes.

To check whether these methods apply to a particular class we have added routines
to the python package \emph{Permuta}. For details see Section~\ref{sec:conclusion}.

\subsection{Mesh patterns}

We end this section with a short introduction to mesh patterns which are
utilised in some of our proofs. A reader
familiar with them might skip directly to Section~\ref{sec:perm_on_grid}.
A \emph{mesh pattern} $p$ is a pair $(\pi, \mathcal{R})$ with
$\pi\in\perms_k$ and $\mathcal{R}\subseteq {\{0,1,\ldots,k\}}^2$.
Pictorially, we represent a mesh pattern in a similar way as a classical pattern,
and we shade, for each $(x,y)\in \mathcal R$, the unit square with bottom left
corner in $(x,y)$. The mesh pattern
$p=(132, \{(0,0), (0,2), (1,2), (2,2), (2, 3), (3,2)\})$ is pictured below.
\begin{center}
    \mpattern{}{3}{1/1,2/3,3/2}{0/0,0/2,1/2,2/2,3/2,2/3}
\end{center}

Intuitively, an occurrence of a mesh pattern $p=(\pi, \mathcal{R})$ in a
permutation $\sigma$ is an occurrence of $\pi$ in $\sigma$ such that,
if we stretch the shading of $\pi$ onto $\sigma$,
$\sigma$ has no point in the shaded region.
For example, we consider the permutation $35142$ and pick two different
occurrences of $132$ in it (see Figure~\ref{fig:mesh-patt-occ}).
We stretch the shading of $p$ for both occurrences.
The one on the left is an occurrence of $p$ since no points of $\sigma$ are in
the shading, however, the right one is not an occurrence of $p$ since the $3$
of $\sigma$ is in the region corresponding to the box $(0,2)$ in $p$.
\begin{figure}[htpb]
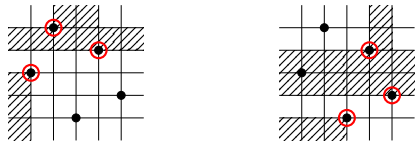

    \centering
    \includestandalone{figures/mesh-patt-occ}
    \caption{Two classical occurrences of $132$ in the permutation $35142$.
        On the left the classical
        occurrence is an occurrence of $p$ whereas the one on the right is not.}%
    \label{fig:mesh-patt-occ}
\end{figure}
Formally, the definition of mesh pattern containment is as follows.
\begin{definition}[\cite{branden2011meshpattern}]
    Let $\pi\in \perms_k$ and $\sigma\in\perms_n$.
    An occurrence of the mesh pattern $p=(\pi, \mathcal{R})$ in a permutation
    $\sigma$ is a subset $\omega$ of the plot of $\sigma$,
    $G(\sigma)=\{(i,\sigma(i)) : i\in \{1,2,\ldots,n\}\}$ such that there are
    order preserving injections
    $\alpha, \beta:\{1,\ldots, k\}\to \{1,\ldots, n\}$
    satisfying the following conditions.
    Firstly, $\omega$ is an occurrence of $\pi$ in the classical sense, \textit{i.e.},
    \begin{enumerate}
        \item[i.] $\omega = \{(\alpha(i), \beta(j)):(i,j)\in G(\sigma)\}$.
    \end{enumerate}
    Define
    $R_{ij} = [\alpha(i)+1, \alpha(i+1)-1]\times[\beta(j)+1,\beta(j+1)-1]$
    for $i,j\in\{1,\ldots,k\}$ where $\alpha(0)=\beta(0)=0$ and
    $\alpha(k+1)=\beta(k+1)=n+1$. Then the second condition is
    \begin{enumerate}
        \item[ii.] if $(i,j)\in\mathcal{R}$ then $R_{ij}\cap G(\sigma)=\emptyset$.
    \end{enumerate}
\end{definition}

If there is an occurrence of $p$ in $\sigma$ we say that $p$ is
\emph{contained in} $\sigma$.
Otherwise, we say that $\sigma$ \emph{avoids}
the mesh pattern $p$.

Unlike for classical patterns, it can occur that $\Av(p) = \Av(q)$ for two different
mesh patterns, $p$, $q$. For instance the mesh patterns $(21, \emptyset)$ and
$(21, \{ (1,0), (1,1), (1,2) \})$ have the same avoiding permutations, since a permutation
has an inversion if and only if it has a descent. Many of these \emph{coincidences}
are captured by the Shading Lemma \cite[Lemma 11]{shl}.

Throughout the paper, we will use the Shading Lemma to argue that the occurrence
of a classical pattern implies the occurrence of a mesh pattern. For example, in
the proof of Lemma~\ref{lem:uperm_subset_av} we argue that an occurrence of
$2314$ implies an occurrence of
\[ \mpattern{}{4}{1/2,2/3,3/1,4/4}{0/1,2/0}.\]
The argument goes as follows. Let $\sigma$ be a permutation with an occurrence of
$2314$. We can shade the box $(2,0)$ by replacing the $1$ of the occurrence by
the bottomost point in that box. The box $(0,1)$ can also be shaded by replacing
the $2$ of the occurrence with the leftmost point in that cell.

%% file: figures/pattern.tex
\begin{tikzpicture}[scale=.3,baseline=(current bounding box.center)]
    \foreach \x in {1,...,9} {
        \draw[ultra thin] (\x,0)--(\x,10); 
        \draw[ultra thin] (0,\x)--(10,\x); 
    }
    \draw[fill=black] (1,1) circle (5pt);
    \draw[fill=black] (2,9) circle (5pt);
    \draw[fill=black] (3,4) circle (5pt);
    \draw[fill=black] (4,5) circle (5pt);
    \draw[fill=black] (5,3) circle (5pt);
    \draw[fill=black] (6,2) circle (5pt);
    \draw[fill=black] (7,6) circle (5pt);
    \draw[fill=black] (8,7) circle (5pt);
    \draw[fill=black] (9,8) circle (5pt);

    \draw[red, thick] (2,9) circle (10pt);
    \draw[red, thick] (4,5) circle (10pt);
    \draw[red, thick] (5,3) circle (10pt);
    \draw[red, thick] (6,2) circle (10pt);
\end{tikzpicture}

%% file: figures/sums.tex
\begin{tikzpicture}[scale=.3,baseline=(current bounding box.center)]
    \foreach \x in {1,...,5} {
        \draw[ultra thin] (\x,0)--(\x,6); 
        \draw[ultra thin] (0,\x)--(6,\x); 
    }
    \draw[fill=black] (1,1) circle (5pt);
    \draw[fill=black] (2,2) circle (5pt);
    \draw[fill=black] (3,3) circle (5pt);
    \draw[fill=black] (4,5) circle (5pt);
    \draw[fill=black] (5,4) circle (5pt);
    \node at (3,-1) {$123\oplus 21$};
    \begin{scope}[shift={(8,0)}]
        \foreach \x in {1,...,5} {
            \draw[ultra thin] (\x,0)--(\x,6); 
            \draw[ultra thin] (0,\x)--(6,\x); 
        }
        \draw[fill=black] (1,3) circle (5pt);
        \draw[fill=black] (2,4) circle (5pt);
        \draw[fill=black] (3,5) circle (5pt);
        \draw[fill=black] (4,2) circle (5pt);
        \draw[fill=black] (5,1) circle (5pt);
        \node at (3,-1) {$123\ominus 21$};
    \end{scope}
\end{tikzpicture}

%% file: figures/mesh-patt-occ.tex
    \begin{tikzpicture}[scale=.3,baseline=(current bounding box.center)]
        \foreach \x in {1,...,5} {
            \draw[ultra thin] (\x,0)--(\x,6); 
            \draw[ultra thin] (0,\x)--(6,\x); 
        }
        \draw[fill=black] (1,3) circle (5pt);
        \draw[fill=black] (2,5) circle (5pt);
        \draw[fill=black] (3,1) circle (5pt);
        \draw[fill=black] (4,4) circle (5pt);
        \draw[fill=black] (5,2) circle (5pt);

        \fill[pattern=north east lines] (0,0) rectangle+(1,3);
        \fill[pattern=north east lines] (0,4) rectangle+(6,1);
        \fill[pattern=north east lines] (2,5) rectangle+(2,1);

        \draw[red, thick] (1,3) circle (10pt);
        \draw[red, thick] (2,5) circle (10pt);
        \draw[red, thick] (4,4) circle (10pt);
        \begin{scope}[shift={(12,0)}]
            \foreach \x in {1,...,5} {
                \draw[ultra thin] (\x,0)--(\x,6); 
                \draw[ultra thin] (0,\x)--(6,\x); 
            }
            \draw[fill=black] (1,3) circle (5pt);
            \draw[fill=black] (2,5) circle (5pt);
            \draw[fill=black] (3,1) circle (5pt);
            \draw[fill=black] (4,4) circle (5pt);
            \draw[fill=black] (5,2) circle (5pt);

            \fill[pattern=north east lines] (0,0) rectangle+(3,1);
            \fill[pattern=north east lines] (0,2) rectangle+(6,2);
            \fill[pattern=north east lines] (4,4) rectangle+(1,2);

            \draw[red, thick] (3,1) circle (10pt);
            \draw[red, thick] (4,4) circle (10pt);
            \draw[red, thick] (5,2) circle (10pt);
        \end{scope}
    \end{tikzpicture}

%% file: 2-encoding.tex
\section{Encoding permutations on grid}\label{sec:perm_on_grid}

A letter $\sigma_i$ in a permutation $\sigma$ is called a \emph{left-to-right minimum}
if
$\sigma_j>\sigma_i$ for all $j< i$.
We denote with $\Avn(B)$ the permutations in $\Av(B)$ with exactly $n$ left-to-right
minima.
For a coarser representation, take a permutation $\sigma$ in $\Avn(B)$ and
place the left-to-right minima on the main diagional of a $n \times n$ grid,
and the remaining points into the cells of the grid with respect to their
relative positions. We then replace the points in each cell by the permutation
they are forming in this cell.
This is called the the \emph{staircase encoding} of
$\sigma$ and is denoted $\SE(\sigma)$.
Figure~\ref{fig:staircase_encoding} shows the staircase
encoding of the permutation $659817432$. As permutations
contained
in cells in the same row or same column can interleave
in multiple ways, the
staircase encoding is not an injective map.
For example, the permutations $659814327$ and
$659718432$ both have the staircase encoding shown in
Figure~\ref{fig:staircase_encoding}~(c).

\begin{figure}[htpb]
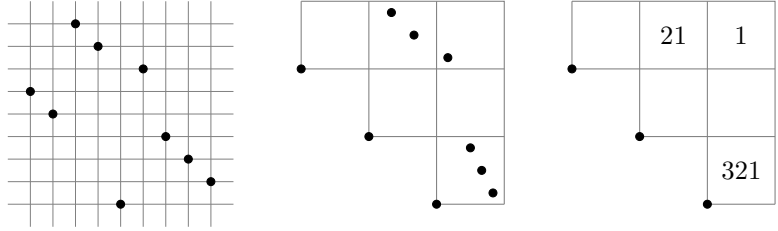

    \centering
    \includestandalone{figures/stair_enc}
    \caption{
        (a) The plot of $\sigma=659817432$.
        (b) The permutation $\sigma$ drawn on the staircase grid.
        (c) The staircase encoding of $\sigma$.
    }%
    \label{fig:staircase_encoding}
\end{figure}

By construction, the staircase encoding only uses the cells above and to the
right of the left-to-right minima.
We define the \emph{staircase grid} $B_n$ as the set of cells of
the staircase encoding of a permutation with $n$ left-to-right minima. The cells
are indexed using matrix coordinates,
\textit{i.e.}, $B_n=\{(i,j): 1\leq i\leq n \text{ and }i\leq j \leq n\}$.
We say that $B_0$ is the empty staircase
grid which corresponds to the staircase encoding of $\varepsilon$, the empty
permutation.
\cite{bean2015pattern} used the staircase grid to enumerate
$\Av(123)$ and $\Av(132)$. We briefly review these in terms of our refined
staircase encoding.

A cell in the staircase encoding of a permutation that avoids $123$
contains a permutation avoiding $12$, since any occurrence of $12$
together with one of the left-to-right minima would give an occurrence of
$123$. Moreover, the presence of a point in a cell forces other cells to be
empty.
For example, in
the encoding of Figure~\ref{fig:staircase_encoding}, we have the staircase
encoding of the $123$ avoiding permutation $659817432$.
As the cell $(1, 3)$
contains a point, the cell $(2, 2)$ must be
empty if the encoding is one of a permutation avoiding $123$.
These constraints are symmetric and
can be represented as a graph, where the cells of $B_n$ are the vertices
and there is an edge between every pair of cells that cannot both contain a
point of the permutation. This graph is called the \emph{up-core} of $B_n$.

\begin{definition}[Definition 4.3 in~\cite{bean2015pattern}]
    Let $n\geq 0$ be an integer. The \emph{up-core} of $B_n$ is the labelled
    undirected graph
    $\UU(B_n)$ with vertex set $B_n$ and edges between $(i,j)$ and $(k,\ell)$ if
    $i>k$, $j<\ell$.
\end{definition}

If a permutation avoids $132$, we get similar restrictions on the staircase
encoding. First, every cell avoids $21$ for a similar reason as above.
Second, some pairs of cells cannot both contain a point.
Those restrictions are also described by a graph called the
\emph{down-core}.
\begin{definition}[Definition 4.3 in~\cite{bean2015pattern}]
    Let $n\geq 0$ be an integer. The \emph{down-core} of $B_n$ is the labelled
    undirected graph
    $\D(B_n)$ with vertex set $B_n$ and edges between $(i,j)$ and $(k,\ell)$ if
    $i<k$, $j<\ell$ and the rectangle
    $\{i, i+1, \ldots k\}\times\{j,j+1, \ldots, \ell\}$ is a subset of $B_n$.
\end{definition}

See Figure~\ref{fig:up-core} for examples of $\UU(B_4)$ and $\D(B_4)$. We say
that a cell of the staircase encoding is \emph{active} if it contains a
non-empty permutation. From the construction of the graphs, we can see that
the set of active cells of the staircase encoding of a permutation in
$\Av(123)$ (resp.\ $\Av(132)$) is an independent set of $\UU(B_n)$ (resp.\
$\D(B_n)$).
The image under
the staircase encoding of $\Av(123)$ is the set of staircase encodings that are
independent sets of $\UU(B_n)$, where the permutations in every cell avoid $12$.

\begin{figure}[htpb]
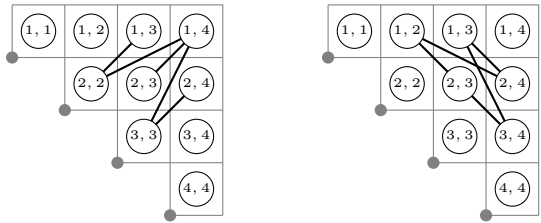

    \centering
    \includestandalone{figures/up-and-down-core}
    \caption{The up-core $\UU(B_4)$ on the left and
        the down-core $\D(B_4)$ on the right.}%
    \label{fig:up-core}
\end{figure}

In order to formalize our previous statement, we introduce \emph{weighted
    independent sets}, an independent set where a
weight is given to each of its vertices. In this paper, the
weights will always be permutations.
We denote by $\WI(G, S)$ the set of all weighted independents sets
of a graph $G$
where the weights are permutations from the set $S$.
Using this notation, we have that
\[\SE(\Avn(123))\subseteq \WI(\UU(B_n), \Avp(12))\]
and
\[\SE(\Avn(132))\subseteq \WI(\D(B_n), \Avp(21)).\]

The $i$-th \emph{row} of a permutation consists of the points with values
between the value of the $i$-th left-to-right minima and the $(i+1)$-st
left-to-right minima of the permutation.
Avoiding $123$ forces rows of the permutation to be \emph{decreasing}.
This means that for two cells $(i,j)$ and $(i,k)$ with $j<k$ the points in
$(i,j)$ are above the points in $(i,k)$, \textit{i.e.}, larger in value.
A decreasing row is pictured in Figure~\ref{fig:row_inc}.

\begin{figure}[htpb]
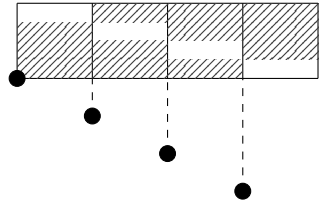

    \centering
    \includestandalone{figures/dec_row}
    \caption{A typical decreasing row.
        There are no points in the shaded regions.}%
    \label{fig:row_inc}
\end{figure}

For a $132$ avoiding permutation, the rows of the permutation are
\emph{increasing}, \textit{i.e.},
for a pairs of cells $(i,j)$ and $(i,k)$ with $j<k$
the points in $(i,j)$ are below the points
in $(i,k)$, \textit{i.e.}, lower in value.

The $j$-th \emph{column} of a permutation consists of the points with index
between the indices of the $j$-th and $(j+1)$-st left-to-right minima of
a permutation.
In a similar manner as above,
we say that the $j$-th column  is \emph{increasing}
(resp.\ \emph{decreasing}) if for
each pair of cells $(i,j)$ and $(k,j)$ with $i>k$ the points in $(i,j)$ are
on the left (resp.\ right) of the points in $(k,j)$.
The columns of a $123$ avoiding permutation are decreasing while the columns of
a $132$ avoiding permutation are increasing.

As mentioned before, the staircase encoding is not an injective map since
many permutations can have the same staircase encoding.
However, by restricting to the set of permutations with increasing
(resp.~decreasing) rows and columns the staircase encoding is an injection.
The inverse of the staircase encoding
restricted to permutations with increasing (resp.~decreasing) rows and columns
is $\dperm$ (resp.~$\uperm$).

\begin{definition}\label{def:uperm-dperm}
    For a staircase encoding $E$, we define
    \begin{itemize}
        \item $\uperm(E)$ as the permutation $\sigma$ with decreasing rows and
              columns such that $\SE(\sigma)=E$.
        \item $\dperm(E)$ as the permutation $\sigma$ with increasing rows and
              columns such that $\SE(\sigma)=E$.
    \end{itemize}
\end{definition}
Both $\uperm$ and $\dperm$ are injective maps from the set of staircase
encodings to the set of all permutations.
Lemma~\ref{lem:se-inverses} follows from the definition.
\begin{lemma}\label{lem:se-inverses}
    The maps $\SE\circ\uperm$ and $\SE\circ\dperm$ are the identity on the set
    of all staircase encodings.
\end{lemma}

\begin{remark}
    Formally, the staircase encoding is a map from the set of all
    permutations to the set of staircase grids filled with
    permutations.
    However, throughout the paper we consider the restriction of $\SE$ to
    a smaller set such that the restriction is a bijection to its image.
    Hence, when the context is clear (as in the theorem below),
    $\SE$ might refer to a restriction of the staircase encoding.
\end{remark}

\begin{theorem}[Lemma~2.2 in \cite{bean2015pattern}]\label{thm:123_132}
    The map $\SE$ is a bijection between $\Avn(123)$ and the weighted
    independent sets $\WI(\UU(B_n), \Avp(12))$. It is also a bijection between $\Avn(132)$
    and the weighted independent sets $\WI(\D(B_n), \Avp(21))$
\end{theorem}

By Theorems~2.4 and~3.3 from~\cite{bean2015pattern} we know that the number of
independent sets of size $k$ in $\UU(B_n)$, or $\D(B_n)$, is given by the
coefficient
of $x^n y^k$ in the generating function  $\gfU(x,y)$ that satisfies
\begin{equation}
    \label{eq:fxy}
    \gfU(x,y) = 1 + x\gfU(x,y) + \frac{xy\gfU{(x,y)}^2}{1-y(\gfU(x,y)-1)}.
\end{equation}

If we substitute $y$ with $\frac{x}{1-x}$ into $\gfU(x,y)$, we obtain
the generating function where the coefficient of
$x^n$ is the number of $123$ avoiding permutation of size $n$.

\begin{corollary}
    The generating function for $\Av(123)$ and $\Av(132)$ is
    $\gfU\left(x, \frac{x}{1-x} \right)$.
\end{corollary}

%% file: figures/stair_enc.tex
\begin{tikzpicture}[scale=0.3]
    \foreach \x in {1,...,9}{%
        \draw[ultra thin, gray] (\x,0)--(\x,10); 
        \draw[ultra thin, gray] (0,\x)--(10,\x); 
    }
    \draw[fill=black] (1,6) circle (5pt);
    \draw[fill=black] (2,5) circle (5pt);
    \draw[fill=black] (3,9) circle (5pt);
    \draw[fill=black] (4,8) circle (5pt);
    \draw[fill=black] (5,1) circle (5pt);
    \draw[fill=black] (6,7) circle (5pt);
    \draw[fill=black] (7,4) circle (5pt);
    \draw[fill=black] (8,3) circle (5pt);
    \draw[fill=black] (9,2) circle (5pt);

    \begin{scope}[shift={(12,0)}]
        \draw[gray] (1, 10)--(10,10)--(10,1);
        \foreach \y in {1,4,7}{%
            \draw[thin, gray] (8-\y, 10)--(8-\y,\y)--(10,\y);
            \draw[fill=black] (8-\y, \y) circle (5pt);
        }
        \draw[fill=black] (5,9.5) circle (5pt);
        \draw[fill=black] (6,8.5) circle (5pt);
        \draw[fill=black] (7.5,7.5) circle (5pt);
        \draw[fill=black] (8.5,3.5) circle (5pt);
        \draw[fill=black] (9.0,2.5) circle (5pt);
        \draw[fill=black] (9.5,1.5) circle (5pt);
    \end{scope}

    \begin{scope}[shift={(24,0)}]
        \draw[gray] (1, 10)--(10,10)--(10,1);
        \foreach \y in {1,4,7}{%
            \draw[thin, gray] (8-\y, 10)--(8-\y,\y)--(10,\y);
            \draw[fill=black] (8-\y, \y) circle (5pt);
        }
        \node at (5.5, 8.5) {$21$};
        \node at (8.5, 8.5) {$1$};
        \node at (8.5, 2.5) {$321$};
    \end{scope}
\end{tikzpicture}

%% file: figures/up-and-down-core.tex
\begin{tikzpicture}[yscale=-1, scale=0.7]
    \foreach \x in {1,...,4} {
        \draw[fill=gray, gray] (\x,\x) circle (3pt);
        \draw[gray] (\x,0)--(\x,\x)--(5,\x);
    }
    \draw[gray] (1,0)--(5,0)--(5,4);
    \foreach \y in {2,3}
    	\pgfmathtruncatemacro{\yu}{\y+1}
    	\pgfmathtruncatemacro{\yd}{\y-1}
    	\foreach \x in {\y,...,3}
    		\pgfmathtruncatemacro{\xu}{\x+1}
    		\foreach \yy in {1,...,\yd} \foreach \xx in {\xu,...,4}
                \draw[thick] (\x+0.5,\y-0.5)--(\xx+0.5,\yy-0.5);

    \foreach \y in {1,...,4} \foreach \x in {\y,...,4}
    	{\node[draw=black, fill=white, circle, inner sep=0.07em] at (\x+0.5,\y-0.5) {\tiny $\y, \x$};}

    \begin{scope}[shift={(6,0)}]
         \foreach \x in {1,...,4} {
             \draw[fill=gray, gray] (\x,\x) circle (3pt);
             \draw[gray] (\x,0)--(\x,\x)--(5,\x);
         }
         \draw[gray] (1,0)--(5,0)--(5,4);

         \foreach \y in {1,...,2}
         \pgfmathtruncatemacro{\yu}{\y+1}
        	\foreach \x in {\yu,...,3}
                     \pgfmathtruncatemacro{\xu}{\x+1}
                     \foreach \yy in {\yu,...,\x} \foreach \xx in {\xu,...,4}
                         \draw[thick] (\x+0.5,\y-0.5)--(\xx+0.5,\yy-0.5);

         \foreach \y in {1,...,4} \foreach \x in {\y,...,4}
         	{\node[draw=black, fill=white, circle, inner sep=0.07em] at (\x+0.5,\y-0.5) {\tiny $\y, \x$};}
    \end{scope}
\end{tikzpicture}

%% file: figures/dec_row.tex
\begin{tikzpicture}[yscale=-1, scale=1]
    \draw (1,0)--(5,0);
    \draw (1,1)--(5,1);
    \foreach \x in {1,...,5} {
        \draw (\x,0)--(\x, 1);
    }
    \fill[pattern color = black!75, pattern=north east lines] (2,0.25)
        rectangle +(-1,0.75);
    \fill[pattern color = black!75, pattern=north east lines] (2,0.25)
        rectangle +(3,-0.25);
    \fill[pattern color = black!75, pattern=north east lines] (3,0.5)
        rectangle +(-1,0.5);
    \fill[pattern color = black!75, pattern=north east lines] (3,0.5)
        rectangle +(2,-0.25);
    \fill[pattern color = black!75, pattern=north east lines] (4,0.75)
        rectangle +(-1,0.25);
    \fill[pattern color = black!75, pattern=north east lines] (4,0.75)
        rectangle +(1,-0.25);

    \draw[fill=black] (1,1) circle (3pt);
    \foreach \x in {2,3,4} {
        \draw[dashed] (\x,0.5+\x/2)--(\x, 1);
        \draw[fill=black] (\x,0.5+\x/2) circle (3pt);
    }
\end{tikzpicture}

%% file: 3-len3to4.tex
\section{Going from size 3 to size 4 patterns}\label{sec:3to4}
As seen in  Section~\ref{sec:perm_on_grid}, avoiding  the pattern  $123$
creates restrictions on which pairs of cells in the staircase grid can
contain points of the permutation. These restrictions were encoded by
the up-core graph. The same restrictions are
enforced on the set
of active cells of the staircase encoding for permutation avoding
$2314$ or $3124$.
The two
patterns are pictured in Figure~\ref{fig:ru_cu}(a). 
In this figure, the black points can be thought of as left-to-right minima of
the permutation and red points as points in the cells of the staircase grid.
Avoiding either
of those patterns ensures that two cells
connected by up-core
edges cannot be active simultaneously.
Moreover, the pattern $2314$, the \emph{row-up pattern}, denoted $r_u$,
forces the rows to be decreasing. Similarly, the pattern $3124$, the
\emph{column-up pattern}, denoted $c_u$, forces
the columns to be decreasing.

\begin{figure}[htpb]
    \centering
    \includestandalone{figures/ru_cu_patt}
    \caption{%
        (a) The  row-up pattern  on the  left and
        the  column-up pattern  on the right.
        (b) The  row-down pattern  on the  left and
        the  column-down pattern  on the right.
    }\label{fig:ru_cu}
\end{figure}

The down-core restrictions can also be enforced using size 4 patterns.
To do so, we introduce the \emph{row-down pattern}
$2413$, denoted $r_d$, and the \emph{column-down pattern} $3142$, denoted $c_d$.
As for the
up-core, thinking of the black points as the left-to-right minima of the permutation
and the red points as points in cells
(see Figure~\ref{fig:ru_cu}(b)), 
we can see this results in the same constraints as in the down-core.
Moreover, these patterns force rows and columns to be
increasing.
The above discussion is formalized in the next two lemmas.

\begin{lemma}\label{lem:row_col_interleaving}
    Let $\sigma$ be a permutation. Then
    \begin{enumerate}
        \item the rows of $\sigma$ are decreasing if $\sigma\in\Av(r_u)$,%
              \label{sublem:row_dec}
        \item the columns of $\sigma$ are decreasing if $\sigma\in\Av(c_u)$,%
              \label{sublem:col_dec}
        \item the rows of $\sigma$ are increasing if $\sigma\in\Av(r_d)$,%
              \label{sublem:row_inc}
        \item the columns of $\sigma$ are increasing if $\sigma\in\Av(c_d)$.%
              \label{sublem:col_inc}
    \end{enumerate}
\end{lemma}
\begin{proof}
    We only prove~\ref{sublem:row_dec}.\ since the other cases can be handled
    similarly.
    Let $\sigma$ be a permutation and
    suppose that one of the rows is not decreasing.
    This row has at least two active cells,
    therefore $\sigma$ has at least two left-to-right minima.
    Hence, when drawn on a staircase grid,
    this row contains two cells $A$ and $B$ such that $B$ is to the right of $A$,
    and $B$ contains a point higher than a point in $A$. These two points
    together with the left-to-right minima to the left of the columns
    containing $A$ and $B$
    form an occurrence of $r_u$ in $\sigma$.
\end{proof}

\begin{lemma}\label{lem:up_down_edge_const}
    Let $\sigma$ be a permutation with $n$ left-to-right minima and $C$ be the
    set of active cells of the staircase encoding of $\sigma$.
    Then $C$ is an independent set of
    \begin{enumerate}
        \item $\UU(B_n)$ if $\sigma\in \Av(r_u)\cup \Av(c_u)$,%
              \label{sublem:up-edges}
        \item $\D(B_n)$ if $\sigma\in \Av(r_d)\cup \Av(c_d)$.%
              \label{sublem:down-edges}
    \end{enumerate}
\end{lemma}
\begin{proof}
    It is sufficient to show~\ref{sublem:up-edges}.,
    which we do by proving the contrapositive: suppose that
    two active cells of the staircase encoding of a permutation $\sigma$ are
    connected by an edge of $\UU(B_n)$.
    Hence, one of the cells is above and to the right of the
    other. Moreover, since they are in distinct rows and distinct columns of
    $B_n$,
    there exist three left-to-right minima as shown on
    Figure~\ref{fig:up-edges-cells}.

    \begin{figure}[htpb]
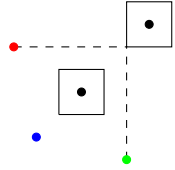

        \centering
        \includestandalone{figures/up-edge-cells}
        \caption{Two cells connected by an edge of the up-core. The blue, the
            red and the green points are distinct left-to-right minima.}%
        \label{fig:up-edges-cells}
    \end{figure}

    The red point and the blue point together with the two points in the active
    cells
    form a $c_u$ pattern. Replacing the red point by the green one yields an
    occurrence of the pattern $r_u$. Hence, $\sigma$ is not in the union $\Av(r_u)\cup \Av(c_u)$.
\end{proof}

In sections~\ref{sec:inf-up-down} to~\ref{sec:rd_2134}, we study different
combinations of the patterns $r_u$, $c_u$, $r_d$ and $c_d$. This leads to
different graphs, weights and constraints on the rows and columns of the
staircase grid.
With each combination of patterns, we describe a set of patterns $P$ that can be
added to the basis while keeping the structural properties of the class that we
need for enumeration. Even when not specified explicitly, we assume throughout the
paper that the empty permutation is not in the set $P$.
The results are presented in order of increasing complexity, with each section
introducing a new tool that is used to build different structural
descriptions and generating function arguments.
Table~\ref{tab:results-overview} presents an
overview of the results in the upcoming sections.
The notation $\bstrip{P}$ that appears in the table is introduced in
Definition~\ref{def:strip}.

\begin{table}[htpb]
    \centering
    \renewcommand{\arraystretch}{1.5}
    \begin{tabular}{l | m{5cm} | l}
        Permutation classes
         & Conditions on the set $P$
         & Enumeration result
        \\ \hline
        $\Av(r_u, c_u, 1\oplus P)$
         & $P$ is skew-indecomposable
         & Corollary~\ref{cor:gf_upcore}
        \\
        $\Av(r_d, c_d, 1\oplus P)$
         & $P$ is sum-indecomposable
         & Corollary~\ref{cor:gf_downcore}
        \\
        $\Av(r_u, c_u, r_d, c_d, 1\oplus P)$
         & No condition on $P$
         & Corollary~\ref{cor:gf_rdcdrucu}
        \\
        $\Av(r_u, c_u, c_d, 1\oplus P)$
         & $P$ is skew-indecomposable
         & Corollary~\ref{cor:gf_rucupi}
        \\
        $\Av(r_d, c_d, c_u, 1\oplus P)$
         & $\bstrip{P}$ is sum-indecomposable                                %
         & Corollary~\ref{cor:gf_rdcdpi}
        \\
        $\Av(r_d, c_u, 1\oplus P)$
         & $P$ is skew-indecomposable and $\bstrip{P}$ is sum-indecomposable
         & Corollary~\ref{cor:gf_rdcu}
        \\
        $\Av(r_d, 2134, P)$
         & $P$ satisfies conditions described in Section~\ref{sec:rd_2134}
         & Corollary~\ref{cor:rd_2134}
        \\
        $\Av(r_u, 2143, P)$
         & $P$ satisfies conditions described in Section~\ref{sec:ru_2143}
         & Corollary~\ref{cor:ru_2143}
    \end{tabular}
    \caption{Overview of the classes we cover in the upcoming sections.}%
    \label{tab:results-overview}
\end{table}

%% file: figures/ru_cu_patt.tex
    \begin{tikzpicture}[scale=.3,baseline=(current bounding box.center)]
            \foreach \x in {1,...,4} {
                    \draw[ultra thin] (\x,0)--(\x,5); 
                    \draw[ultra thin] (0,\x)--(5,\x); 
            }
            \draw[fill=black] (1,2) circle (5pt);
            \draw[fill=red] (2,3) circle (5pt);
            \draw[fill=black] (3,1) circle (5pt);
            \draw[fill=red] (4,4) circle (5pt);
            \node at (2.5,-1) {$r_u=2314$};
            \begin{scope}[shift={(7,0)}]
                \foreach \x in {1,...,4} {
                    \draw[ultra thin] (\x,0)--(\x,5); 
                    \draw[ultra thin] (0,\x)--(5,\x); 
                }
                \draw[fill=black] (1,3) circle (5pt);
                \draw[fill=black] (2,1) circle (5pt);
                \draw[fill=red] (3,2) circle (5pt);
                \draw[fill=red] (4,4) circle (5pt);
                \node at (2.5,-1) {$c_u= 3124$};
            \end{scope}
            \node at (6.5, -3) {(a)};

            \begin{scope}[shift={(20,0)}]
            \foreach \x in {1,...,4} {
                    \draw[ultra thin] (\x,0)--(\x,5); 
                    \draw[ultra thin] (0,\x)--(5,\x); 
            }
            \draw[fill=black] (1,2) circle (5pt);
            \draw[fill=red] (2,4) circle (5pt);
            \draw[fill=black] (3,1) circle (5pt);
            \draw[fill=red] (4,3) circle (5pt);
            \node at (2.5,-1) {$r_d=2413$};
            \begin{scope}[shift={(7,0)}]
                \foreach \x in {1,...,4} {
                    \draw[ultra thin] (\x,0)--(\x,5); 
                    \draw[ultra thin] (0,\x)--(5,\x); 
                }
                \draw[fill=black] (1,3) circle (5pt);
                \draw[fill=black] (2,1) circle (5pt);
                \draw[fill=red] (3,4) circle (5pt);
                \draw[fill=red] (4,2) circle (5pt);
                \node at (2.5,-1) {$c_d= 3142$};
            \end{scope}
            \node at (6.5, -3) {(b)};
            \end{scope}
    \end{tikzpicture}

%% file: figures/up-edge-cells.tex
\begin{tikzpicture}[scale=0.3]
    \draw (2,2)--(2,4)--(4,4)--(4,2)--(2,2);
    \draw (5,5)--(5,7)--(7,7)--(7,5)--(5,5);
    \draw[fill=black] (6,6) circle (5pt);
    \draw[fill=black] (3,3) circle (5pt);

    \draw[dashed] (0, 5)--(5,5);
    \draw[dashed] (5, 0)--(5,5);

    \draw[red, fill=red] (0,5) circle (5pt);
    \draw[blue, fill=blue] (1,1) circle (5pt);
    \draw[green, fill=green] (5,0) circle (5pt);
\end{tikzpicture}

%% file: 4-up-and-down-cores.tex
\section{Weighted independent sets of the up-core and the down-core}\label{sec:inf-up-down}

Lemmas~\ref{lem:row_col_interleaving} and~\ref{lem:up_down_edge_const}
say that every permutation in $\Av(r_u, c_u)$
can be constructed by first taking an independent set of the up-core of
a staircase grid, and weighting the cells with permutations in
$\Av(r_u, c_u)$.
We will show how this can be used to enumerate the class $\Av(r_u,c_u)$ and many of its
subclasses.
We first show an auxiliary result used in the proof of
our main results.

\begin{lemma}\label{lem:uperm_subset_av}
    Let $P$ be a set of skew-indecomposable permutations.
    Then for $n\geq 1$,
    \begin{equation}\label{eq:uperm_subset_av}
        \uperm(\WI(\UU(B_n), \Avp(r_u,c_u,P)))\subseteq \Avn(r_u, c_u, 1\oplus P).
    \end{equation}
\end{lemma}
\begin{proof}
    First, for a skew-indecomposable permutation $\pi$, we will show
    \begin{equation}\label{eq:first_inc}
        \uperm(\WI(\UU(B_n), \Avp(\pi))) \subseteq \Avn(1\oplus\pi).
    \end{equation}
    Assume that $\sigma \in \uperm(\WI(\UU(B_n), \Avp(\pi)))$
    contains $1\oplus\pi$. Then, a rectangular region of the staircase grid of $\sigma$
    contains $\pi$.
    As the set of active cells is an independent set of the up-core, the
    rows and columns are decreasing, and $\pi$ is skew-indecomposable, $\pi$
    occurs in a single cell. This is a contradiction, since the weights are from
    $\Avp(\pi)$.

    We will complete the proof by showing that
    \begin{equation*}
        \uperm(\WI(\UU(B_n),\Avp(r_u, c_u))) \subseteq \Avn(r_u, c_u).
    \end{equation*}
    Let $\sigma$ be a permutation in $\uperm(\WI(\UU(B_n),\Avp(r_u, c_u)))$.
    If $\sigma$ contains either $r_u$ or $c_u$ then $\sigma$ contains one of
    the mesh patterns $m_1$ or $m_2$ in Figure~\ref{fig:mesh-pattern}, by the Shading Lemma.
    If the cell $(1,0)$ of $m_1$ (resp.\ cell $(0,1)$ of $m_2$)
    contains a point then, by picking the leftmost point in that region,
    the permutation $\sigma$ contains an occurrence of $m_2$ (resp.\ $m_1$)
    that is
    below (resp.\ to the left) of the occurrence we are considering.
    Repeat this argument on the new occurrence. As $\sigma$ is finite, we will
    repeat a finite number of times until we find an occurrence of $m_3$
    or $m_4$ in Figure~\ref{fig:mesh-pattern}.

    \begin{figure}[htb]
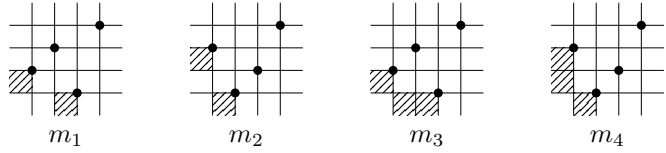

        \centering
        \includestandalone{figures/mp1_mp2}
        \caption{A permutation that contains $r_u$ or $c_u$
            contains $m_1$ or $m_2$, as well as $m_3$ or $m_4$.}%
        \label{fig:mesh-pattern}
    \end{figure}

    Assume $\sigma_{i_1}\sigma_{i_2}\sigma_{i_3}\sigma_{i_4}$ is an occurrence of
    $m_3$ in $\sigma$. Either $\sigma_{i_1}$ and $\sigma_{i_3}$ are
    both left-to-right minima in $\sigma$, or both are not left-to-right minima
    in $\sigma$.
    If they are left-to-right minima of $\sigma$ then $\sigma_{i_2}$ and
    $\sigma_{i_4}$ are in different columns of the staircase grid, and moreover
    different rows as rows are decreasing. This implies that two connected
    cells in $\UU(B_n)$ are active, contradicting the fact that an independent set
    was used. Therefore, $\sigma_{i_1}$ and $\sigma_{i_3}$ are not
    left-to-right minima of $\sigma$.
    There is, therefore, a point
    $(k, \sigma_k)$ with $k < i_1$ and $\sigma_k < \sigma_{i_3}$. This new point
    together with the original occurrence is an occurrence of
    $1 \oplus r_u$. As $r_u$ is skew-indecomposable, this contradicts
    Equation~\eqref{eq:first_inc}.

    Hence, we have shown that $\sigma$ avoids $m_3$.
    A similar argument shows that $m_4$ is also avoided, and hence $r_u$ and
    $c_u$ are avoided by $\sigma$.
\end{proof}

\begin{theorem}\label{thm:inf_upcore}
    Let $P$ be a set of skew-indecomposable permutations.
    Then $\SE$ is a bijection between
    $\Avn(r_u, c_u, 1\oplus P)$ and
    $\WI(\UU(B_n), \Avp(r_u, c_u, P))$.
\end{theorem}
\begin{proof}
    Let $\sigma$ be a permutation in $\Avn(r_u, c_u,1\oplus P)$.
    By Lemma~\ref{lem:up_down_edge_const}, the active
    cells of the staircase encoding of $\sigma$ form an independent set of
    $\UU(B_n)$, and the subpermutations in each cell of the staircase
    encoding are in $\Av(r_u, c_u, P)$. Hence,
    \[\SE(\Avn(r_u,c_u,1\oplus P))\subseteq \WI(\UU(B_n), \Avp(r_u, c_u, P))\]

    By applying $\SE$ on both sides of Equation~\eqref{eq:uperm_subset_av} in
    Lemma~\ref{lem:uperm_subset_av}, we get by Lemma~\ref{lem:se-inverses}
    \[\WI(\UU(B_n), \Avp(r_u, c_u, P))\subseteq \SE(\Avn(r_u,c_u,1\oplus P)).\]
    Hence,
    \[\SE(\Avn(r_u,c_u,1\oplus P))= \WI(\UU(B_n), \Avp(r_u, c_u, P)).\]
    Since permutations avoiding $r_u$ and $c_u$ have decreasing rows and columns,
    the map $\SE$ is injective when restricted to $\Avn(r_u,c_u,1\oplus P)$.
    Therefore, $\SE$ is a bijection between
    $\Avn(r_u, c_u, 1\oplus P)$ and
    $\WI(\UU(B_n), \Avp(r_u, c_u, P))$.
\end{proof}

The following corollary shows how to
compute the generating function for any basis covered by the theorem.

\begin{corollary}\label{cor:gf_upcore}
    Let $P$ be a set of skew-indecomposable permutations and $A(x)$ be the
    generating function of $\Av(r_u,  c_u,  P)$.
    Then $\Av(r_u, c_u, 1\oplus P)$ is enumerated by $\gfU(x,A(x)-1)$,
    where $\gfU(x, y)$ is the generating function in Equation~\eqref{eq:fxy}.
\end{corollary}

\begin{proof}
    By Theorem~\ref{thm:inf_upcore}, $\Av(r_u, c_u, 1\oplus P)$ is in 1-to-1
    correspondence with
    \[\bigsqcup_{n\geq0}\WI(\UU(B_n), \Avp(r_u, c_u, P)).\]
    Moreover, the size of the permutation obtained is the number of
    left-to-right minima added to the sizes of the weights of the
    independent set.
    This implies $\gfU(x,A(x)-1)$
    is the generating function for $\Av(r_u, c_u, 1\oplus P)$.
\end{proof}

Corollary~\ref{cor:gf_upcore} can be used to compute
the generating function of $\Av(2314, 3124)$, that was first enumerated
by \cite{MR1754331}. The generating function $A(x)$ for
$\Av(2314, 3124)$ satisfies
\begin{equation}\label{eq:schroeder_rec}
    A(x) = \gfU(x, A(x)-1).
\end{equation}
Solving gives
\begin{equation}\label{eq:schroeder_gf}
    A(x) = \frac{3-x-\sqrt{1-6x+x^2}}{2},
\end{equation}
which is the generating function for the
large Schr\"oder numbers, which can be found in the Online Encyclopedia of
Integer sequences~\cite{oeis} as sequence \oeis{A006318}.

Corollary~\ref{cor:gf_upcore} can also be used to enumerate the subclass
$\Av(2314, 3124, 1234)$,
first enumerated by \cite{MR3633253}. In this case,
the cells of the independent sets are filled with permutations in
$\Av(2314, 3124, 123)=\Av(123)$.
Since the generating function of the latter class is
$\frac{1-\sqrt{1-4x}}{2x}$, 
the generating function of $\Av(2314, 3124, 1234)$ is
\[ \gfU\left(x, \frac{1-\sqrt{1-4x}}{2x} -1 \right).\] 

There are three different skew-indecomposable permutations of size 3. Those
permutations are $123$, $132$ and $213$.
Therefore, Theorem~\ref{thm:inf_upcore} gives a structural
description of $\Av(2314, 3124)$ and subclasses obtained by also avoiding any
subset of $\{1234, 1243, 1324\}$. This gives $8$ classes with bases consisting
of only size four patterns. Since, there are $13$ skew-indecomposable
permutations of size 4, the theorem gives structural description of $2127$
bases
\footnote{This is the number of bases after removing redundancies.}
that contain size 4 and 5 patterns.

Theorem~\ref{thm:inf_upcore} extends the number of permutation classes that
the up-core describes. A similar method can be used for the down-core to
enumerate permutation classes beyond $\Av(132)$.
\begin{lemma}\label{lem:dperm_subset_av}
    Let $P$ be a set of sum-indecomposable permutations. Then for $n\geq 1$
    \[\dperm(\WI(\D(B_n), \Avp(r_d, c_d, P)))\subseteq \Avn(r_d, c_d, 1\oplus P).\]
\end{lemma}
\begin{theorem}\label{thm:inf_downcore}
    Let $P$ be a set of sum-indecomposable permutations. Then
    $\SE$ is a bijection between
    $\Avn(r_d, c_d, 1\oplus P)$ and
    $\WI(\D(B_n), \Avp(r_d, c_d, P))$.
\end{theorem}
The proofs are left to the reader as they are similar to the proofs of
Lemma~\ref{lem:uperm_subset_av} and Theorem~\ref{thm:inf_upcore}.
Corollary~\ref{cor:gf_downcore} follows naturally from
Theorem~\ref{thm:inf_downcore}.

\begin{corollary}\label{cor:gf_downcore}
    Let $P$ be a set of sum-indecomposable permutations and $A(x)$ be the
    generating  function of  $\Av(r_d,  c_d,  P)$.
    Then $\Av(r_d, c_d, 1\oplus P)$ is enumerated by $\gfU(x,A(x)-1)$.
\end{corollary}

As a consequence of the previous corollary, $A(x)$, the generating function of
$\Av(2413, 3142)$, first enumerated by \cite{MR1754331},
also satisfies Equation~\eqref{eq:schroeder_rec} and is given by
Equation~\eqref{eq:schroeder_gf}. This result can enumerate 8 classes with
bases consisting of size four patterns and many more if we consider
longer patterns.

It is worth noting that any subclass of $\Av(2413, 3142)$ (as well as the class
itself) contains finitely
many simple permutations and can be enumerated using a more
general method called the substitution decomposition, described in
\cite{substitutiondecomposition}. \cite{autosubdecomp} extended the method to
allow for random sampling. We outline briefly in Section~\ref{sec:conclusion}
how the the structural description introduced in this paper
can be used to randomly sample in permutation classes, including many with
infinitely many simple permutations.

%% file: figures/mp1_mp2.tex
\begin{tikzpicture}[scale=.3,baseline=(current bounding box.center)]
    \foreach \x in {1,...,4} {
        \draw[ultra thin] (\x,0)--(\x,5); 
        \draw[ultra thin] (0,\x)--(5,\x); 

    }
    \fill[pattern=north east lines] (0, 1) rectangle +(1,1);
    \fill[pattern=north east lines] (2, 0) rectangle +(1,1);
    \draw[fill=black] (1,2) circle (5pt);
    \draw[fill=black] (2,3) circle (5pt);
    \draw[fill=black] (3,1) circle (5pt);
    \draw[fill=black] (4,4) circle (5pt);
    \node at (2.5,-1) {$m_1$};

    \begin{scope}[shift={(8,0)}]
        \foreach \x in {1,...,4} {
            \draw[ultra thin] (\x,0)--(\x,5); 
            \draw[ultra thin] (0,\x)--(5,\x); 

        }
        \fill[pattern=north east lines] (0, 2) rectangle +(1,1);
        \fill[pattern=north east lines] (1, 0) rectangle +(1,1);
        \draw[fill=black] (1,3) circle (5pt);
        \draw[fill=black] (2,1) circle (5pt);
        \draw[fill=black] (3,2) circle (5pt);
        \draw[fill=black] (4,4) circle (5pt);
        \node at (2.5,-1) {$m_2$};

    \end{scope}

    \begin{scope}[shift={(16,0)}]
        \foreach \x in {1,...,4} {
            \draw[ultra thin] (\x,0)--(\x,5); 
            \draw[ultra thin] (0,\x)--(5,\x); 

        }
        \fill[pattern=north east lines] (0, 1) rectangle +(1,1);
        \fill[pattern=north east lines] (1, 0) rectangle +(1,1);
        \fill[pattern=north east lines] (2, 0) rectangle +(1,1);
        \draw[fill=black] (1,2) circle (5pt);
        \draw[fill=black] (2,3) circle (5pt);
        \draw[fill=black] (3,1) circle (5pt);
        \draw[fill=black] (4,4) circle (5pt);
        \node at (2.5,-1) {$m_3$};
    \end{scope}

    \begin{scope}[shift={(24,0)}]
        \foreach \x in {1,...,4} {
            \draw[ultra thin] (\x,0)--(\x,5); 
            \draw[ultra thin] (0,\x)--(5,\x); 

        }
        \fill[pattern=north east lines] (0, 2) rectangle +(1,1);
        \fill[pattern=north east lines] (0, 1) rectangle +(1,1);
        \fill[pattern=north east lines] (1, 0) rectangle +(1,1);
        \draw[fill=black] (1,3) circle (5pt);
        \draw[fill=black] (2,1) circle (5pt);
        \draw[fill=black] (3,2) circle (5pt);
        \draw[fill=black] (4,4) circle (5pt);
        \node at (2.5,-1) {$m_4$};

    \end{scope}
\end{tikzpicture}

%% file: 5-updown-core.tex
\section{Inflating the updown-core}\label{sec:UDRC-core}

In the previous section, we enumerated $\Av(2314, 3124)$ and $\Av(2413, 3142)$
and many of their subclasses. However, the intersection of the two classes,
namely the subclass $\Av(2314, 3124, 2413, 3142)$, cannot
be enumerated using the theorems so far.
Together those patterns put stricter constraints on the staircase encoding that
we have not encountered yet.
In this section, we combine different
graphs to represent these constraints and, in particular, give
the enumeration of $\Av(2314, 3124, 2413, 3142)$ and many of its subclasses.
Again, this class and any subclasses could be enumerated using the substitution
decomposition. However, the techniques used in this section are an important
stepping stone for the upcoming sections.

To represent the new constraint,
we introduce the \emph{column-edges} that connect cells in the same
column of a grid and the \emph{row-edges} that connect cells in the same
row. More formally:
\begin{definition}\label{def:col_row_edges}
    \item
    \begin{itemize}
        \item The \emph{column-core} graph $\C(B_n)$ is the graph whose
              vertices are the cells of $B_n$
              and where there is an edge between cells
              $(i,j)$ and $(k,\ell)$ if $i\neq k$ and $j=\ell$.
        \item The \emph{row-core} graph $\R(B_n)$ is the graph  whose
              vertices are the cells of $B_n$ and
              where there is an edge between cells
              $(i,j)$ and $(k,\ell)$ if $i= k$ and $j\neq\ell$.
    \end{itemize}
\end{definition}

We combine the edges of the four graphs $\UU(B_n)$, $\D(B_n)$, $\C(B_n)$ and
$\R(B_n)$ in the natural way.
For example, the graph $\UDC(B_n)$ has the cells of $B_n$ as vertices and
the edges of $\UU(B_n)$, $\D(B_n)$ and $\C(B_n)$. Figure~\ref{fig:mixing-core}
shows $\UDC(B_4)$.

\begin{figure}[htpb]
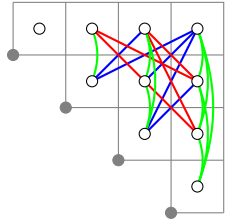

    \centering
    \includestandalone{figures/udc4}
    \caption{The graph $\UDC(B_4)$.
        The up-edges are blue, the down-edges are red and the column-edges are
        green.}%
    \label{fig:mixing-core}
\end{figure}

As we did for the up-core and the down-core with
Lemma~\ref{lem:up_down_edge_const}, we describe sufficient conditions for the
staircase encoding of a permutation to respect the constraints enforced by the
column-core and the row-core.

\begin{lemma}\label{lem:col_row_edges_const}
    Let $\sigma$ be a permutation with $n$ left-to-right minima and $C$ be the
    set of active cells of the staircase encoding of $\sigma$.
    Then $C$ is an independent set of
    \begin{enumerate}
        \item $\R(B_n)$ if $\sigma\in \Av(r_u, r_d)$,%
              \label{sublem:row-edges}
        \item $\C(B_n)$ if $\sigma\in \Av(c_u, c_d)$.%
              \label{sublem:col-edges}
    \end{enumerate}
\end{lemma}

\begin{proof}
    By Lemma~\ref{lem:row_col_interleaving} if $\sigma\in \Av(r_u, r_d)$,
    the rows of $\sigma$ are both increasing and decreasing. Therefore,
    there is at most one active cell in each row,
    and so the active cells correspond to an independent set in $\R(B_n)$.
    The proof of~\ref{sublem:col-edges}.\ is similar.
\end{proof}

\begin{theorem}
    Let $P$ be a set of permutations. Then $\SE$ is a bijection between
    $\Avn(r_d, c_d, r_u, c_u, 1\oplus P)$ and
    $\WI(\UDRC(B_n),\Avp(r_d, c_d, r_u, c_u, P))$.
\end{theorem}
\begin{proof}
    By Lemmas~\ref{lem:up_down_edge_const} and~\ref{lem:col_row_edges_const}
    we have that
    \begin{equation*}
        \SE(\Av(r_d,c_d, r_u, c_u, 1\oplus P)) \subseteq
        \WI(\UDRC(B_n), \Avp(r_d, c_d, r_u, c_u, P)).
    \end{equation*}

    Let $\sigma$ be a permutation in $\uperm(\WI(\UDRC(B_n),\Avp(r_d, c_d, r_u, c_u, P)))$.
    As any independent set of $\UDRC(B_n)$ is an independent set of $\UU(B_n)$
    and
    \[\Av(r_d, c_d, r_u, c_u, P) \subseteq \Av(r_u, c_u)\]
    we have that
    \begin{equation*}
        \uperm(\WI(\UDRC(B_n),\Avp(r_d, c_d, r_u, c_u,P)))
        \subseteq
        \uperm(\WI(\UU(B_n),\Avp(r_u, c_u)))
    \end{equation*}
    Therefore, by Lemma~\ref{lem:uperm_subset_av} it follows that
    $\sigma \in \Av(r_u, c_u)$. By observing that $\uperm$
    and $\dperm$ are equivalent when building from an independent set of
    $\RC(B_n)$, it follows
    from a symmetric argument and Lemma~\ref{lem:dperm_subset_av} that
    $\sigma \in \Av(r_d, c_d)$.

    Assume that
    $\sigma$ contains a pattern $1 \oplus \pi$, where $\pi \in P$. Without loss
    of generality, we can assume that the $1$ in the occurrence is a
    left-to-right minimum in $\sigma$. Then $\pi$ occurs in the rectangular
    region of cells north-east of the minimum. By the constraints of
    $\UDRC(B_n)$, there can be at most one active cell in this region, and so
    $\pi$ is contained in this cell. This contradicts the fact that active
    cells are filled with permutations avoiding $\pi$. Therefore, $\sigma$
    avoids $1 \oplus P$, and moreover
    $\sigma \in \Avn(r_d,c_d, r_u, c_u, 1\oplus P)$.
    We get
    \[ \uperm(\WI(\UDRC(B_n),\Avp(r_d, c_d, r_u, c_u,P))) \subseteq
        \Avn(r_d,c_d,r_u,c_u,1\oplus P)\]
    or equivalently by applying $\SE$ on both sides
    \[ \WI(\UDRC(B_n),\Avp(r_d, c_d, r_u, c_u,P)) \subseteq
        \SE(\Avn(r_d,c_d,r_u,c_u,1\oplus P)).\]

    Since permutations avoiding $r_u$ and $c_u$ have decreasing rows and columns,
    the map $\SE$ is injective when restricted to $\Avn(r_u,c_u,r_d,c_d,1\oplus P)$.
    Therefore, $\SE$ is a bijection between
    $\Avn(r_u,c_u,r_d,c_d,1\oplus P)$ and the image of that set which is
    $\WI(\UDRC(B_n), \Avp(r_u,c_u,r_d,c_d,P))$.
\end{proof}

Since the graph $\UDRC(B_n)$ is the same graph as the updown-core of $B_n$ defined
in \cite{bean2015pattern}, Lemma~4.13
of the same paper gives
that the number of independent sets of size $k$ in $\UDRC(B_n)$
is given by the coefficient of $x^n y^k$ in the generating function
\[\gfUDRC(x,y) = \frac{1-x}{x^2-xy-2x+1}.\]

As in Section~\ref{sec:inf-up-down}, we get an enumeration result for the
classes
\sloppy ${\Av(r_d, c_d, r_u, c_u, 1\oplus P)}$.
\begin{corollary}\label{cor:gf_rdcdrucu}
    Let $P$ be a set of permutations and $A(x)$ be the
    generating  function of  $\Av(r_d,  c_d,  r_u, c_u, 1\oplus P)$.
    Then $A(x)$ satisfies
    \[A(x) = \gfUDRC(x, B(x)-1)\]
    where $B(x)$ is the generating function of $\Av(r_d, c_d, r_u, c_u, P)$.
\end{corollary}

Solving the equation $A(x) = \gfUDRC(x, A(x)-1)$ gives
\[ A(x) =
    \frac{x^{2}-x-\sqrt{x^{4}-2x^{3}+7x^{2}-6x+1}+1}{2x}\]
which is the generating function for $\Av(2413, 3142, 2314, 3124)$.
This class was first enumerated by \cite{MR1948771}
and the sequence appears on OEIS as \oeis{A078482}.

%% file: figures/udc4.tex
\begin{tikzpicture}[yscale=-1, scale=0.7]
    \foreach \x in {1,...,4} {
        \draw[fill=gray, gray] (\x,\x) circle (3pt);
        \draw[gray] (\x,0)--(\x,\x)--(5,\x);
    }
    \draw[gray] (1,0)--(5,0)--(5,4);
    \foreach \y in {2,3}
    	\pgfmathtruncatemacro{\yu}{\y+1}
    	\pgfmathtruncatemacro{\yd}{\y-1}
    	\foreach \x in {\y,...,3}
    		\pgfmathtruncatemacro{\xu}{\x+1}
    		\foreach \yy in {1,...,\yd} \foreach \xx in {\xu,...,4}
                \draw[thick, color=blue] (\x+0.5,\y-0.5)--(\xx+0.5,\yy-0.5);

     \foreach \y in {1,...,2}
     \pgfmathtruncatemacro{\yu}{\y+1}
            \foreach \x in {\yu,...,3}
                 \pgfmathtruncatemacro{\xu}{\x+1}
                 \foreach \yy in {\yu,...,\x} \foreach \xx in {\xu,...,4}
                     \draw[thick, color=red] (\x+0.5,\y-0.5)--(\xx+0.5,\yy-0.5);

     \foreach \y in {1,2, 3}
     \pgfmathtruncatemacro{\yu}{\y+1}
            \foreach \x in {\yu,...,4}
                 \pgfmathtruncatemacro{\xu}{\x+1}
                 \foreach \yy in {\y,...,\x}
                     \draw[thick, color=green] (\x+0.5,\y-0.5) to[in=-70,out=70] (\x+0.5,\yy-0.5);

     \foreach \y in {1,...,4} \foreach \x in {\y,...,4}
         {\draw[black, fill=white] (\x+0.5,\y-0.5) circle (3pt);}
\end{tikzpicture}

%% file: 6-new-cores.tex
\section{New cores}\label{sec:new-core}

To this point we have considered permutation classes that can be described by
filling the independent sets of the graphs $\UU(B_n)$, $\D(B_n)$ and $\UDRC(B_n)$,
which were first used by~\cite{bean2015pattern} to
enumerate permutation classes avoiding size $3$ patterns. In this section, we
begin to consider new graphs that were not motivated by permutation classes
avoiding size $3$ patterns.

We first consider $\UDC(B_n)$.
The active cells in the staircase encoding of a permutation $\sigma$
avoiding $c_u$ and $c_d$ are an independent set of $\C(B_n)$.
By Lemma~\ref{lem:up_down_edge_const}, they are also an independent set of
$\UU(B_n)$ and $\D(B_n)$. In order to make the filling of independent sets
unique we need that the rows are either increasing or decreasing,
\textit{i.e.}, $\sigma$ avoids $r_u$ or $r_d$.
In this section, we consider additionally avoiding $r_u$, and delay the
discussion of avoiding $r_d$ to Section~\ref{sec:gen_inflation}.

\begin{theorem}\label{thm:inf_cu_cd_ru}
    Let $P$ be a set of skew-indecomposable permutations. Then $\SE$ is a
    bijection between
    $\Avn(r_u,c_u,c_d, 1\oplus P)$ and
    $\WI(\UDC(B_n),\Avp(r_u, c_u, c_d, P))$.
\end{theorem}

\begin{proof}
    Let $\sigma \in \uperm(\WI(\UDC(B_n), \Avp(r_u,c_u,c_d,P)))$.
    As $\UDC(B_n)$ contains the edges of $\UU(B_n)$ we have
    \begin{equation*}
        \uperm(\WI(\UDC(B_n), \Avp(r_u,c_u,c_d,P)))
        \subseteq
        \uperm(\WI(\UU(B_n), \Avp(r_u,c_u,c_d,P)))
    \end{equation*}
    and so by Lemma~\ref{lem:uperm_subset_av}, $\sigma$
    avoids $r_u$, $c_u$, $1\oplus c_d$ and $1\oplus P$. Suppose that $\sigma$
    contains an occurrence of $c_d$, then by the Shading Lemma
    it also has an occurrence of the mesh pattern with the same underlying
    pattern and cells $(0, 2)$ and $(1, 0)$ shaded. Further, the avoidance of
    $r_u$ and $1 \oplus c_d$ imply that there is an occurrence with the cells
    $(0, 1)$ and $(0, 0)$ also shaded. Let
    $\sigma_{i_1}\sigma_{i_2}\sigma_{i_3}\sigma_{i_4}$
    be an occurrence of this
    mesh pattern, shown in Figure~\ref{fig:rd_mp}.
    \begin{figure}[htpb]
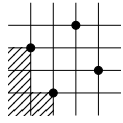

        \centering
        \includestandalone{figures/mp3}
        \caption{The mesh pattern that is contained in the permutation
            $\sigma$
            if it contains an occurrence of $c_d$.}%
        \label{fig:rd_mp}
    \end{figure}
    Both $\sigma_{i_1}$ and
    $\sigma_{i_2}$ are left-to-right minima of $\sigma$. Therefore,
    $\sigma_{i_3}$ and $\sigma_{i_4}$ are in two separate active cells,
    contradicting the fact that the active cells are an independent set of
    $\D(B_n)$ and $\C(B_n)$. Hence, $\sigma \in \Av(r_u, c_u, c_d, 1\oplus P)$
    and
    \[\uperm(\WI(\UDC(B_n), \Avp(r_u,c_u,c_d,P))) \subseteq
        \Avn(r_u,c_u,c_d,1\oplus P).\]
    By applying $\SE$ to both sides we get
    \[\WI(\UDC(B_n), \Avp(r_u,c_u,c_d,P)) \subseteq
        \SE(\Avn(r_u,c_u,c_d,1\oplus P)).\]

    \noindent
    By Lemmas~\ref{lem:up_down_edge_const} and~\ref{lem:col_row_edges_const},
    \[ \SE(\Avn(r_u,c_u,c_d,1\oplus P)) \subseteq
        \WI(\UDC(B_n), \Avp(r_u,c_u,c_d,P)).\]
    Moreover, by
    Lemma~\ref{lem:row_col_interleaving} the rows and columns are decreasing,
    therefore, restricted to $\Avn(r_u,c_u,c_d,1\oplus P)$, $\SE$ is injective
    and a bijection to its image which is
    $\WI(\UDC(B_n), \Av(r_u,c_u,c_d,P))$.
\end{proof}

In order to use the theorem above for enumerative purposes,
we need to find the generating function where the coefficient of $x^n y^k$ is
the number of independent sets of size $k$ in $\UDC(B_n)$.
We prove a slightly more general statement that
tracks the number of rows occupied by the independent set.
Although, not required for the permutation classes discussed in this section
it will be necessary for the results in Section~\ref{sec:gen_inflation}.

\begin{proposition}\label{prop:gf_UDR}
    The number of independent sets of size $k$ occupying $\ell$ rows
    in $\UDC(B_n)$
    is given by the coefficient of $x^n y^k z^\ell$ in the generating function
    \[\gfUDC(x,y,z) = \frac{1-x-xy}{x^2 y -xyz +x^2 - xy -2x + 1}.\]
\end{proposition}
\begin{proof}
    An independent set in $\UDC(B_n)$ can contain an arbitrary number of
    vertices in the topmost row, \textit{i.e.}, vertices of the form $(1,j)$. The
    number of such vertices is called the \emph{degree}.
    If the degree is $0$, then the subgraph induced by the remaining vertices is
    isomorphic to a smaller core $\UDC(B_{n-1})$.
    If the degree is not $0$, let
    \[k=\max\{j:(1,j)\text{ is a vertex of the independent set}\},\]
    \textit{i.e.}, $k$ is the column of the rightmost vertex in the topmost row.
    The independent set cannot contain a vertex $(\ell,m)$ if $1<\ell\leq k$
    or if $\ell=1$ and $m\geq k$. This corresponds to the region shaded in gray
    in Figure~\ref{fig:cont_cu_cd}.

    \begin{figure}[htpb]
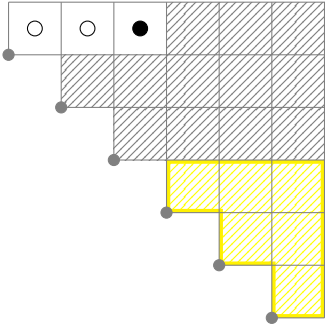

        \centering
        \includestandalone{figures/udc_shading}
        \caption{A staircase grid with an active cell marked by a black point.
            The shaded cells are the cells that cannot be added to make an
            independent set. Cells marked with a circle are disconnected from the
            graph induced by removing the shaded cells.}%
        \label{fig:cont_cu_cd}
    \end{figure}
    Moreover, the vertices $\{(1,j): 1\leq j < k\}$
    share no edges. We can, therefore, chose independently
    if they are in the independent set.

    The graph induced by the remaining vertices $(\ell, m)$
    with $\ell > k$ and $m > k$, form an instance of the graph
    $\UDC(B_{n - k})$. In Figure~\ref{fig:cont_cu_cd} this is the yellow
    region.
    Hence, $\gfUDC(x,y,z)$ satisfies
    \begin{align*}
        \gfUDC(x,y, z) & =
        1 + x\gfUDC(x,y,z) + xyz\gfUDC(x,y,z) + \cdots + x^i
        y{(y+1)}^{i-1}z\gfUDC(x,y,z) + \cdots                                    \\
                       & = 1 + x\gfUDC(x,y) + \frac{xyz\gfUDC(x,y,z)}{1-x(y+1)}. \\
    \end{align*}
    Solving this equation gives the closed form claimed in the proposition.
\end{proof}

As the proof of Theorem~\ref{thm:inf_cu_cd_ru} gives a unique encoding of
the permutation classes $\Av(r_u,c_u,c_d, 1\oplus P)$ we derive the following
corollary to give their enumeration.

\begin{corollary}\label{cor:gf_rucupi}
    Let $P$ be a set of skew-indecomposable permutations and
    $A(x)$ be the generating function of $\Av(r_u, c_u, c_d, 1\oplus P)$.
    Then $A(x)$ satisfies
    \[ A(x) = \gfUDC(x, B(x)-1, 1) \]
    where $B(x)$ is the generating function of $\Av(r_u, c_u, c_d, P)$.
\end{corollary}

%% file: figures/mp3.tex
\begin{tikzpicture}[scale=.3,baseline=(current bounding box.center)]
    \foreach \x in {1,...,4} {
        \draw[ultra thin] (\x,0)--(\x,5); 
        \draw[ultra thin] (0,\x)--(5,\x); 

    }
    \fill[pattern=north east lines] (0, 0) rectangle +(1,1);
    \fill[pattern=north east lines] (1, 0) rectangle +(1,1);
    \fill[pattern=north east lines] (0, 1) rectangle +(1,1);
    \fill[pattern=north east lines] (0, 2) rectangle +(1,1);
    \draw[fill=black] (1,3) circle (5pt);
    \draw[fill=black] (2,1) circle (5pt);
    \draw[fill=black] (3,4) circle (5pt);
    \draw[fill=black] (4,2) circle (5pt);
\end{tikzpicture}

%% file: figures/udc_shading.tex
\begin{tikzpicture}[yscale=-1, scale=0.7]
    \fill[pattern color = gray, pattern=north east lines] (4, 0) rectangle +(3,1);
    \fill[pattern color = gray, pattern=north east lines] (2, 1) rectangle +(5,1);
    \fill[pattern color = gray, pattern=north east lines] (3, 2) rectangle +(4,1);

    \highlightRegion{yellow}{(4,3)--(7,3)--(7,6)--(6,6)--(6,5)--(5,5)--(5,4)--(4,4)}
    \drawGrid{6}
    \draw[fill=black] (3.5, 0.5) circle (4pt);
    \draw[] (2.5, 0.5) circle (4pt);
    \draw[] (1.5, 0.5) circle (4pt);

\end{tikzpicture}

%% file: 7-generalization.tex
\section{Generalizing the fillings}\label{sec:gen_inflation}

As Section~\ref{sec:new-core} considers the basis $\{r_u, c_u, c_d\}$,
one could hope that we can handle $\{r_d, c_d, c_u\}$ similarly.
Unfortunately, the proof of
Theorem~\ref{thm:inf_cu_cd_ru} relies heavily on the
fact that $c_d$ is skew-indecomposable. To repeat the argument for
$\{r_d, c_d, c_u\}$, one would need $c_u$ to
be sum-indecomposable which is not the case.
However, tracking an additional
statistic on the independent set allows us to enumerate these classes.
Even if the classes considered in the section could be enumerated using the
substitution decomposition, the tracking we are about to introduce will also be
used in sections~\ref{sec:ru_cd}, \ref{sec:rd_2134} and~\ref{sec:ru_2143}
for many classes that cannot be enumerated with the substitution decomposition.

If we consider the
staircase encoding of $\sigma\in\Avn(r_d, c_d, c_u)$, we have from
Lemmas~\ref{lem:up_down_edge_const} and~\ref{lem:col_row_edges_const}
that the set of active cells form an independent set
of $\UDC(B_n)$ and that the rows are increasing.
Hence, if one cell in a row contains an occurrence of $312$ and
is not the rightmost non-empty cell of the row,
an occurrence of $c_u=3124$ is created (see Figure~\ref{fig:splitting_cu}).
If an active cell contains $312$ but is the rightmost cell in the row then no
$c_u$ pattern is created since the cells above and to the right are empty.
Hence, the staircase encoding of a permutation avoiding the pattern $c_u$
avoids $312$ in its cells except the rightmost active cell of each row.

\begin{figure}[htpb]
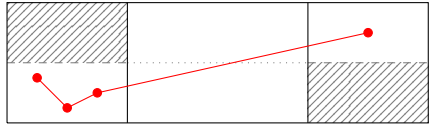

    \centering
    \includestandalone{figures/crossing_3124}
    \caption{An occurrence of $c_u=3124$ spanning across two cells.}%
    \label{fig:splitting_cu}
\end{figure}

In order to describe the set of staircase encodings of permutations in
$\Av(r_d,c_d,c_u)$ we enrich our definition of a weighted independent set
to a \emph{weighted labelled independent set}.
We first define a labelling function $\rl$ on an independent set.
This function maps a vertex $v$ of an independent set $I$ to a label in the
set $\{y,z\}$.
\begin{equation*}
    \rl(v, I) =
    \begin{cases}
        y & \text{if $v$ is not the rightmost cell of $I$ in its row}, \\
        z & \text{otherwise}.
    \end{cases}
\end{equation*}
\begin{definition}
    For a graph $G$ on the staircase grid,
    we define $\WI_{\rl}(G, Y, Z)$ as the set of weighted independent sets of $G$
    where the  weight of a vertex  $v$ in an  independent set $I$ is an element of $Y$ if
    $\rl(v,I)=y$ or an element of $Z$ if $\rl(v,I)=z$.
\end{definition}

We introduce an operation that removes the last value of a
permutation if this value is the maximum of the permutation.
\begin{definition}\label{def:strip}
    Let $\sigma$ be a permutation. We define the permutation $\bstrip{\sigma}$
    as
    \[ \bstrip{\sigma} = \begin{cases}
            \alpha & \textrm{if $\sigma=\alpha\oplus1$ for a permutation $\alpha$}, \\
            \sigma & \textrm{otherwise.}
        \end{cases} \]
\end{definition}

For example, $\bstrip{3124}=312$ and $\bstrip{1432}=1432$.
For a set of patterns $P$, let $\bstrip{P} = \{\bstrip{\pi}: \pi\in P\}$.

\begin{theorem}\label{thm:inf_rd_cd_cu}
    Let $P$ be a set of permutations such that $\bstrip{P}$ contains only
    sum-indecomposable permutations. Then $\SE$ is a bijection between
    $\Avn(r_d,c_d,c_u,1\oplus P)$ and
    $\WI_{\rl}(\UDC(B_n), \Avp(312,\bstrip{P}), \Avp(r_d,c_d,c_u,P))$.
\end{theorem}

\begin{proof}
    From Lemmas~\ref{lem:up_down_edge_const}, and~\ref{lem:col_row_edges_const} and
    the discussion above, we have that
    \[\SE(\Avn(r_d,c_d,c_u,1\oplus P)) \subseteq
        \WI_{\rl}(\UDC(B_n), \Avp(312,\bstrip{P}), \Avp(r_d,c_d,c_u,P)).\]

    To show the reverse inclusion,
    we partition $P$ into two sets depending on whether the permutation
    ends with its maximum or not.
    We set
    \[P_1 = \{\pi \in P: \bstrip{\pi}= \pi\} \text{ and }
        P_2 = \{ \pi\in P: \bstrip{\pi}\neq\pi\}.\]
    Both $\Av(312, \bstrip{P})$ and $\Av(r_d, c_d, c_u, P)$ are subclasses of
    $\Av(r_d, c_d, P_1)$, so we get
    \begin{align*}
         & \dperm(\WI_{\rl}(\UDC(B_n), \Avp(312,\bstrip{P}), \Avp(r_d,c_d,c_u,P)))
        \\ \subseteq&
        \dperm(\WI(\UDC(B_n), \Avp(r_d,c_d,P_1)))
        \\ \subseteq&
        \dperm(\WI(\D(B_n), \Avp(r_d,c_d,P_1))).
    \end{align*}
    As $P_1$ contains only sum-indecomposable permutations, by
    Lemma~\ref{lem:dperm_subset_av}, we have that
    \begin{equation*}
        \dperm(\WI_{\rl}(\UDC(B_n), \Avp(312,\bstrip{P}), \Avp(r_d,c_d,c_u,P)))
        \subseteq
        \Avn(r_d, c_d, 1\oplus P_1).
    \end{equation*}

    We also need to show that $c_u$ and $1\oplus P_2$ are avoided.
    Let $\sigma$ be a permutation in
    $\dperm(\WI_{\rl}(\UDC(B_n), \Avp(312,\bstrip{P}), \Avp(r_d,c_d,c_u,P)))$.
    We first show that for $\pi\in P_2\cup\{c_u\}$,
    $1\oplus\pi$ does not occur. By the hypothesis, we know that
    $\pi=\alpha\oplus 1$, with $\alpha$ sum-indecomposable.
    If $1\oplus\pi$ is contained in $\sigma$, then $\pi$ is fully
    contained in a rectangular region of the grid.
    In such a region, the active cells of the encoding are in the same
    row. Hence, $\pi$ is contained in a single row.
    Since the rows are increasing and $\alpha$ is sum-indecomposable, the only
    way to split the occurrence is to have an occurrence of $\alpha$ in a cell
    and an occurrence of $1$ in a cell to the right. This is not allowed by the
    way the vertices can be weighted.
    Hence, it is contained in a single cell which is also
    forbidden. Therefore, by contradiction, $1\oplus\pi$ is avoided for any
    $\pi$ in $P_2\cup\{c_u\}$.

    In particular $1\oplus c_u$ is avoided.
    Using the Shading Lemma to shade the
    cells $(0,2)$ and $(1,0)$, the avoidance of $r_d$ to shade $(0,1)$ and the
    avoidance of $1\oplus c_u$ to shade $(0,0)$, we see that if $c_u$ is
    contained in $\sigma$ then $\sigma$ contains an occurrence of the mesh pattern
    \[\mpattern{}{4}{1/3,2/1,3/2,4/4}{0/2,1/0,0/1,0/0}.\]
    An occurrence of this mesh pattern violates either the column-edges
    constraint or the up-edges constraints.
    Thus $c_u$ is avoided, and we have
    \begin{equation*}
        \dperm(\WI_{\rl}(\UDC(B_n), \Avp(312,\bstrip{P}), \Avp(r_d,c_d,c_u,P)))
        \subseteq
        \Avn(r_d, c_d, c_u,1\oplus P),
    \end{equation*}
    and, by Lemma~\ref{lem:se-inverses},
    \begin{equation*}
        \WI_{\rl}(\UDC(B_n), \Avp(312,\bstrip{P}), \Avp(r_d,c_d,c_u,P))
        \subseteq
        \SE(\Avn(r_d, c_d, c_u,1\oplus P)).
    \end{equation*}

    By Lemma~\ref{lem:row_col_interleaving}, the rows and columns of a permutation
    in $\Avn(r_d, c_d, c_u,1\oplus P)$ are decreasing and therefore, restricted
    to this set, the map $\SE$ is injective. Hence, $\SE$ is a bijection between
    $\Avn(r_d, c_d, c_u,1\oplus P)$ and
    \[
        \SE(\Avn(r_d,c_d,c_u,1\oplus P))=
        \WI_{\rl}(\UDC(B_n), \Avp(312,\bstrip{P}), \Avp(r_d,c_d,c_u,P)).\tag*{\qed}
    \]\renewcommand{\qedsymbol}{}
\end{proof}

Recall that $\gfUDC(x,y,z)$ is the generating function of independent sets of
$\UDC(B_n)$ where $x$ tracks the size of the grid, $y$ the
size of the independent set, and $z$ the number of rows of the
independent set. Therefore, $\gfUDC(x,y, \frac{z}{y})$ is the
generating function where $y$ tracks the number of cells labelled $y$ by
the labelling function $\rl$ and $z$ tracks the number of cells labelled $z$.
Let $C_1$ and $C_2$ be two permutation classes enumerated by
$A(x)$ and $B(x)$. Let $F(x)$ be the generating function for the number of
weighted independent sets where the cells labelled $y$ are weighted with a
non-empty permutation from $C_1$ and the cells labelled $z$ are weighted with a
non-empty permutation from $C_2$.
Then
\[ F(x) = \gfUDC\left(x, A(x)-1, \frac{B(x)-1}{A(x)-1}\right).\]
This leads to the following enumeration result:

\begin{corollary}\label{cor:gf_rdcdpi}
    Let $P$ be a set of permutations such that $\bstrip{P}$ contains only
    sum-indecomposable permutations, and
    $A(x)$ be the generating function of $\Av(r_d, c_d, c_u, 1\oplus P)$.
    Then $A(x)$ satisfies
    \[ A(x) = \gfUDC\left(x, C(x)-1, \frac{B(x)-1}{C(x)-1}\right) \]
    where $B(x)$ is the generating function of $\Av(r_d, c_d, c_u, P)$ and
    $C(x)$ is the generating function of $\Av(312, \bstrip{P})$.
\end{corollary}

As an example, we derive $A(x)$, the generating function of
$\Av(2413, 3142, 3124)$ which was first derived
by \cite{MR3633253} and appears in the OEIS as
\oeis{A033321}.
Since the basis is $\{r_d, c_d, c_u\}$, $B(x) = A(x)$. Moreover, $\Av(312)$ is
enumerated by the Catalan numbers and
$C(x)= \frac{1 - \sqrt{1 - 4x}}{2x}$. 
Hence, $A(x)$ satisfies
\[ A(x)=\gfUDC\left(x, A(x)-1, \frac{C(x)-1}{A(x)-1} \right).\]
The equation can be solved to get the explicit form of the generating function.

%% file: figures/crossing_3124.tex
\begin{tikzpicture}[scale=0.8]
    \draw[gray, dashed] (0,1)--(2,1);
    \draw[gray, dashed] (5,1)--(7,1);
    \draw[gray, dotted] (2,1)--(5,1);
    \fill[pattern color=gray, pattern=north east lines] (0,1) rectangle +(2,1);
    \fill[pattern color=gray, pattern=north east lines] (5,0) rectangle +(2,1);
    \draw (0,0)--(7, 0)--(7,2)--(0,2)--(0,0);
    \draw (2,0)--(2,2);
    \draw (5,0)--(5,2);

    \draw[red, fill=red] (0.5,0.75) circle (2pt);
    \draw[red, fill=red] (1.0,0.25) circle (2pt);
    \draw[red, fill=red] (1.5,0.50) circle (2pt);
    \draw[red, fill=red] (6.0,1.50) circle (2pt);

    \draw[red] (0.5,0.75)--(1.0,0.25)--(1.5,0.50)--(6.0,1.50);
\end{tikzpicture}

%% file: 8-back-to-2patts.tex
\section{Avoiding the row-down and column-up patterns}\label{sec:ru_cd}

In this section, we consider permutation classes described by weighted
independent sets of the graphs $\UD(B_n)$.
This corresponds to removing one of the three patterns from the results in
sections~\ref{sec:new-core} and~\ref{sec:gen_inflation}.

\begin{theorem}
    Let $P$ be a set of skew-indecomposable permutations such that
    all permutations of $\bstrip{P}$ are sum-indecomposable.
    Then $\SE$ is a bijection between $\Avn(r_d, c_u, 1\oplus P)$ and
    $\WI_{\rl}(\UD(B_n), \Avp(312,\bstrip{P}), \Avp(r_d,c_u,P))$.
\end{theorem}
\begin{proof}
    By Lemma~\ref{lem:row_col_interleaving}, the map $\SE$ is injective when
    restricted to $\Avn(r_d, c_u, 1\oplus P)$ as each permutation in this set
    has decreasing columns and increasing rows.
    Therefore, to show that $\SE$ is
    the claimed bijection, it is sufficient to show that
    \[
        \SE(\Avn(r_d,c_u,1\oplus P)) =
        \WI_{\rl}(\UD(B_n), \Avp(312,\bstrip{P}), \Avp(r_d,c_u,P)).
    \]

    By Lemma~\ref{lem:up_down_edge_const} any encoding in
    $\SE(\Avn(r_d,c_u,1\oplus P))$ is an independent set of $\UD(B_n)$.
    Moreover, since the rows are increasing, all active cells but the
    rightmost of each row avoid $312$ and $\bstrip{P}$ as discussed at the
    beginning of Section~\ref{sec:gen_inflation}.
    This implies
    \[
        \SE(\Avn(r_d,c_u,1\oplus P)) \subseteq
        \WI_{\rl}(\UD(B_n), \Avp(312,\bstrip{P}), \Avp(r_d,c_u,P)).
    \]
    Take $I$ in $\WI_{\rl}(\UD(B_n), \Avp(312,\bstrip{P}), \Avp(r_d,c_u,P))$.
    We consider $\sigma$, the permutation obtained from $I$ by building the
    permutation with decreasing columns and increasing rows. We show that
    $\sigma$ is in $\Avn(r_d, c_u, 1\oplus P)$.
    We start by showing that $\sigma$ avoids $1\oplus r_d$, $1\oplus
        c_u$ and $1\oplus P$.
    In an occurrence of any of these patterns in $\sigma$, we can assume that
    the $1$ is a left-to-right minimum. Hence, we have to show that $r_d$,
    $c_u$ and $P$ are avoided in the square formed by the set of cells that
    are north and east of a left-to-right minimum. Let $\pi$ be any pattern
    in $\{r_d, c_u\}\cup P$. We know that $\pi$ is skew-indecomposable and that
    $\bstrip{\pi}$ is sum-indecomposable.

    Assume that $\pi$ occurs in a square region of the cells that are north and
    east of a left-to-right minimum. We consider the rightmost column in the
    region that contains a point of the occurrence of $\pi$. In this column, we
    consider the topmost cell that contains such a point.
    In Figure~\ref{fig:no_spanning_rdcu}, this cell is colored blue.
    By construction, the gray region is empty.
    \begin{figure}[htpb]
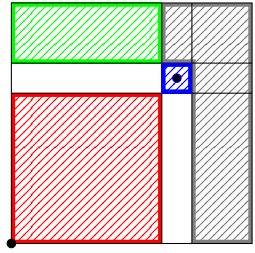

        \centering
        \includestandalone{figures/no_splitting}
        \caption{Decomposition of an occurrence of $\pi$ in a square
            region.}%
        \label{fig:no_spanning_rdcu}
    \end{figure}
    Moreover, there is no point of the permutation
    in the green region (resp.\ the red region) because the set of active
    cells is an independent set of $\D(B_n)$ (resp.\ $\UU(B_n)$).
    Since the columns are decreasing,
    if the occurrence contains any point in a cell below the blue one,
    $\pi$ is skew-decomposable.
    Hence, the only active cell
    in the column is the blue one. Finally, since the row is increasing and
    $\bstrip{\pi}$ is sum-indecomposable, $\pi$ is either fully contained in the blue
    cell, which is forbidden, or $\bstrip{\pi}$ is fully contained in a cell that is
    not the rightmost active one in the row, which is also forbidden.
    Consequently, $1\oplus r_d$, $1\oplus c_u$ and $1\oplus P$ are avoided.
    Using shading arguments as done in previous proofs, we can show that if $\sigma$
    contains an occurrence of $r_d$ or $c_u$ then $\sigma$ contains an occurrence
    of either of the mesh patterns
    \[\mpattern{}{4}{1/2,2/4,3/1,4/3}{0/0,1/0,2/0,0/1} \qquad \text{or}\qquad
        \mpattern{}{4}{1/3,2/1,3/2,4/4}{0/0,1/0,0/2,0/1}.\]
    An occurrence of either of those pattern violates the edge,
    increasing rows, or decreasing column constraints. Hence, $\sigma$ is in
    $\Avn(r_d,c_u,1\oplus P)$ and
    \[\SE(\Avn(r_d,c_u,1\oplus P)) \supseteq
        \WI_{\rl}(\UD(B_n), \Avp(312,\bstrip{P}), \Avp(r_d,c_u,P)).\]
    This proves that the image of $\Avn(r_d,c_u,1\oplus P)$ under $\SE$ is the set
    claimed in the theorem.
\end{proof}

In order to enumerate these classes, we first enumerate the independent
sets of $\UD(B_n)$ while keeping track of the number of rows.

\begin{proposition}\label{prop:gf_UD}
    The number of independents set of size $k$ in $\UD(B_n)$ occupying
    $\ell$ rows
    is given by the coefficient of $x^n y^k z^\ell$ in the
    generating function that satisfies
    \[ \gfUD(x,y,z) = 1 + x\gfUD(x,y, z) + D(x,y,z)\gfUD(x,y,z),\]
    where
    \[ D(x,y,z) = \frac{xyz(xy^2z-x+1)}{(xyz+x-1)(xy+x-1)}. \]
\end{proposition}
\begin{proof}
    We consider the topmost row of $B_n$.  If it contains no vertex of the
    independent set, then we are looking at the independent set in a smaller
    core graph and it contributes $x\gfUD(x,y,z)$ to $\gfUD(x,y,z)$.

    If the topmost row contains vertices of the independent set, we consider its
    rightmost vertex. The vertex is highlighted in blue in
    Figure~\ref{fig:ud_core_decomposition}. The cells in the yellow region
    do not contain any vertices of the independent set because they are
    connected to the blue cell by an edge. The vertices of the independent set
    are, therefore, in the hook formed by the white and blue cells, and the
    pink region. The pink region is completely disconnected from the hook and
    hence the vertices of the independent set in this region
    correspond to an independent set of a smaller core.

    To find $\gfUD$, we need to enumerate the independent set of the hook
    that contains the corner cell of the hook.
    We say that the \emph{leg length} of the hook is the number of cells in the
    horizontal strip.
    If the coefficient
    of $x^n y^k z^\ell$ in $D(x,y,z)$ is the number of such sets of
    $k$ vertices using $\ell$ rows in the hook of leg length $n$, then
    \[\gfUD(x,y,z) = 1 + x\gfUD(x,y,z) + D(x,y,z)\gfUD(x,y,z).\]

    \begin{figure}[htpb]
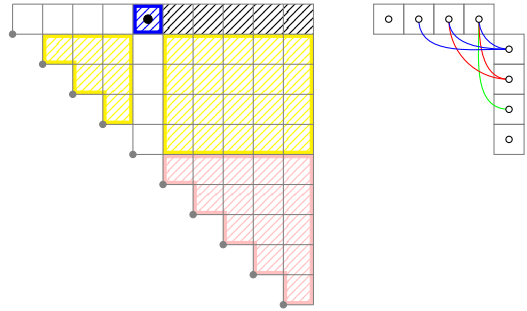

        \centering
        \includestandalone{figures/decomp_ud_core}
        \caption{Decomposition of an independent set according to the rightmost
            vertex in the top row. The picture on the left shows the whole
            staircase.
            The picture on the right shows the induced subgraph of the hook
            with a corner in the blue cell.}%
        \label{fig:ud_core_decomposition}
    \end{figure}

    To find the generating function $D(x,y,z)$ we first notice that any vertices
    that
    we take in the vertical leg add to the row count of the independent set
    while vertices
    in the horizontal leg do not change the row count.
    First, the case where the hook is a single cell
    contributes $xyz$ to
    $D(x,y,z)$.

    Second, if the leg length of the hook is greater than $1$,
    the cells at the end of each
    leg are not connected to the hook by any edges of the graph.
    Hence, we have complete freedom to put
    them in the independent set. Therefore, the second case is
    of the form $(xyz)(xyz+x)(y+1)(\cdots)$. As they are accounted
    for, we completely ignore the corner cell and the two cells at the end of
    the leg and focus on
    enumerating the part of the independent set in the remaining cells.

    If no
    cell of the vertical leg is in the independent set, then any
    cell of the horizontal leg can be in it.
    Otherwise, if there are $i$ cells above
    the topmost active cell of the vertical leg then the $i$ leftmost
    cells of the
    horizontal leg are the only cells from that leg that can be
    in the independent set (see Figure~\ref{fig:hook_gf}).
    \begin{figure}[htpb]
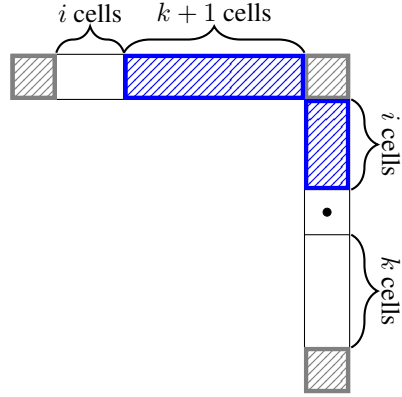

        \centering
        \includestandalone{figures/hook_analyse}
        \caption{The structure of an independent set in the hook. The only
            active cells are in the regions that are not shaded.}%
        \label{fig:hook_gf}
    \end{figure}

    Hence,
    \begin{align*}
        D(x,y,z) & = xyz + (xyz)(xyz+x)(y+1)                                                        \\
                 & \qquad\cdot\left( \frac{1}{1-x(y+1)} + \frac{xyz}{(1-x(y+1))(1-x(1+yz))} \right) \\
                 & = \frac{xyz(xy^2z-x+1)}{(xyz+x-1)(xy+x-1)}
    \end{align*}
    which completes the proof.
\end{proof}

Using the same reasoning as in Section~\ref{sec:gen_inflation} we derive the
following enumeration result:
\begin{corollary}\label{cor:gf_rdcu}
    Let $P$ be a set of skew-indecomposable permutations such that all
    permutations in $\bstrip{P}$ are sum-indecomposable and
    $A(x)$ be the generating of $\Av(r_d, c_u, 1\oplus P)$.
    Then $A(x)$
    satisfies
    \[ A(x) = \gfUD\left(x, C(x)-1, \frac{B(x)-1}{C(x)-1}\right) \]
    where $B(x)$ is the generating function of $\Av(r_d, c_u, P)$
    and
    $C(x)$ is the generating function of $\Av(312, \bstrip{P})$.
\end{corollary}

%% file: figures/no_splitting.tex
\begin{tikzpicture}[scale=0.4]
    \highlightRegion{green}{(0,6)--(5,6)--(5,8)--(0,8)}
    \highlightRegion{red}{(0,0)--(5,0)--(5,5)--(0,5)}
    \highlightRegion{gray}{(6,0)--(8,0)--(8,8)--(5,8)--(5,6)--(6,6)}
    \draw[black] (0,0)--(0,8)--(8,8)--(8,0)--(0,0);
    \draw[black] (0,5)--(8,5);
    \draw[black] (0,6)--(8,6);
    \draw[black] (5,0)--(5,8);
    \draw[black] (6,0)--(6,8);
    \draw[black, fill=black] (0, 0) circle (4pt);
    \highlightRegion{blue}{(5,5)--(5,6)--(6,6)--(6,5)}
    \draw[blue, fill=black] (5.5, 5.5) circle (4pt);
\end{tikzpicture}

%% file: figures/decomp_ud_core.tex
\begin{tikzpicture}[yscale=-1, scale=0.4, region/.style={
        line width=1mm,
        pattern=north east lines,
    }, vertex/.style={
        black,
        fill=white,
    }]
    \begin{scope}
        \highlightRegion{blue}{(5,0) rectangle +(1,1)}
        \draw[black, fill=black] (5.5, 0.5) circle (4pt);
    \end{scope}

    \fill[pattern color=black, pattern=north east lines] (6,0) rectangle +(5,1);

    \begin{scope}
        \highlightRegion{yellow}{(2,1)--(2,2)--(3,2)--(3,3)--(4,3)--(4,4)--(5,4)--(5,1)}
    \end{scope}

    \begin{scope}
        \highlightRegion{yellow}{(6,5)--(6,1)--(11,1)--(11,5)}
    \end{scope}

    \begin{scope}
        \highlightRegion{pink}{(6,5)--(6,6)--(7,6)--(7,7)--(8,7)--(8,8)--(9,8)--(9,9)--(10,9)--(10,10)--(11,10)--(11,5)}
    \end{scope}

    \drawGrid{10}

    \begin{scope}[shift={(12,0)}]
        \foreach \x in {1,...,4} {
            \draw[gray] (\x,0)--(\x,1)--(\x+1,1)--(\x+1,0)--cycle;
            \draw[gray] (5,\x)--(6,\x)--(6,\x+1)--(5,\x+1)--cycle;
        }
        \draw[green] (5.5, 3.5) to[in=90, out=180] (4.5,0.5);
        \draw[red]   (5.5, 2.5) to[in=90, out=180] (4.5,0.5);
        \draw[red]   (5.5, 2.5) to[in=90, out=180] (3.5,0.5);
        \draw[blue]  (5.5, 1.5) to[in=90, out=180] (4.5,0.5);
        \draw[blue]  (5.5, 1.5) to[in=90, out=180] (3.5,0.5);
        \draw[blue]  (5.5, 1.5) to[in=90, out=180] (2.5,0.5);

        \draw[vertex] (1.5, 0.5) circle (3pt);
        \draw[vertex] (2.5, 0.5) circle (3pt);
        \draw[vertex] (3.5, 0.5) circle (3pt);
        \draw[vertex] (4.5, 0.5) circle (3pt);

        \draw[vertex] (5.5, 1.5) circle (3pt);
        \draw[vertex] (5.5, 2.5) circle (3pt);
        \draw[vertex] (5.5, 3.5) circle (3pt);
        \draw[vertex] (5.5, 4.5) circle (3pt);
    \end{scope}
\end{tikzpicture}

%% file: figures/hook_analyse.tex
\begin{tikzpicture}[scale=0.3]
    \draw (0,13)--(0,15)--(15,15)--(15,0)--(13,0)--(13,13)--cycle;
    \draw (13,2)--(15,2);
    \draw (13,7)--(15,7);
    \draw (13,9)--(15,9);
    \draw (13,13)--(15,13);
    \draw (2,13)--(2,15);
    \draw (5,13)--(5,15);
    \draw (13,13)--(13,15);

    \highlightRegion{gray}{(0,13)--(2,13)--(2,15)--(0,15)}
    \highlightRegion{gray}{(13,13)--(15,13)--(15,15)--(13,15)}
    \highlightRegion{gray}{(13,0)--(15,0)--(15,2)--(13,2)}
    \highlightRegion{blue}{(5,13)--(13,13)--(13,15)--(5,15)}
    \highlightRegion{blue}{(13,13)--(15,13)--(15,9)--(13,9)}
    \draw[fill=black] (14,8) circle (5pt);
    \draw [thick, decorate,decoration={brace,amplitude=10pt},yshift=2pt]
        (2,15) -- (5, 15) node [black,midway,yshift=15pt]
        {$i$ cells};
    \draw [thick, decorate,decoration={brace,amplitude=10pt},yshift=2pt]
        (5,15) -- (13, 15) node [black,midway,yshift=15pt]
        {$k+1$ cells};
    \draw [thick, decorate,decoration={brace,amplitude=10pt},xshift=2pt]
        (15,13) -- (15, 9) node [black,midway,xshift=15pt, rotate=-90]
        {$i$ cells};
    \draw [thick, decorate,decoration={brace,amplitude=10pt},xshift=2pt]
        (15,7) -- (15, 2) node [black,midway,xshift=15pt, rotate=-90]
        {$k$ cells};
\end{tikzpicture}

%% file: 9-the-crazy-one.tex
\section{Avoiding \texorpdfstring{$r_d$}{rd} and
  \texorpdfstring{$2134$}{2134}}\label{sec:rd_2134}

In this section, we consider the pattern $2134$ that has not been considered
yet. As when we analysed $r_u$, $c_u$, $r_d$ and $c_d$, we
consider the two black points of Figure~\ref{fig:2134} as left-to-right minima
of the permutation. Then, we study the effect of the two red points on the
staircase encoding of the permutation, its rows and its columns.
An
occurrence of $2134$ where the two black points are left-to-right minima cannot
have a point in the cell of the leading diagonal of the grid. Hence, the
pattern does not enforce any restrictions on those cells.
However, in the
remaining cells, the pattern $2134$ has the same effect as $123$ has on the
grid $B_{n-1}$ since there are always two left-to-right minima to the left and
below those cells.

\begin{figure}[htpb]
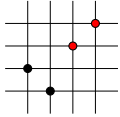

    \centering
    \includestandalone{figures/2134}
    \caption{The pattern 2134.}%
    \label{fig:2134}
\end{figure}

Let $G(B_{n-1})$ be a graph with cells of $B_{n-1}$ as vertices.
Let $S$ be a subset of $B_n$.
For the remainder of the paper, we will make a small abuse of
the definition and say that $S$ is an
independent set of $G(B_{n-1})$ if the set
\[\{(x,y-1) : (x,y)\in S \textrm{ and } x\neq y\}\]
is an independent set of $G(B_{n-1})$.
In other words, a subset of $B_n$ is an independent set of a graph on $B_{n-1}$
if it is an independent set of the graph obtained by overlaying $G(B_{n-1})$ on
$B_n$ as in Figure~\ref{fig:smaller-indep-set}.

\begin{figure}[htpb]
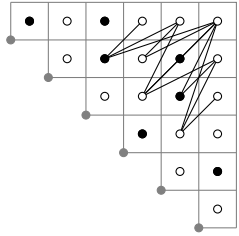

    \centering
    \includestandalone{figures/indep_sub_graph}
    \caption{
        The graph $\UU(B_5)$ on the staircase grid $B_6$. The black vertices form
        an independent set of $\UU(B_5)$ for the grid $B_6$.
    }%
    \label{fig:smaller-indep-set}
\end{figure}

Using similar arguments as we did for
Lemmas~\ref{lem:row_col_interleaving},~\ref{lem:up_down_edge_const}
and~\ref{lem:col_row_edges_const}, we can prove the following lemmas.
\begin{lemma}\label{lem:const_2134}
    Let $\sigma$ be a permutation in $\Avn(2134)$.
    We consider the set $C$ of active cells of the staircase encoding of $\sigma$
    that are not in the main diagonal of the grid. Then
    \begin{itemize}
        \item $C$ is an independent set of $U(B_{n-1})$
        \item cells in $C$ contain decreasing permutations
    \end{itemize}
    Moreover, if we remove the cells of the leading diagonal from $\sigma$,
    the rows and columns are decreasing.
\end{lemma}

\begin{lemma}\label{lem:const_2134_rd}
    Let $\sigma$ be a permutation in $\Avn(2134, r_d)$.
    The set of active cells of the staircase encoding of $\sigma$ is
    and independent set of $\R(B_{n-1})$.
\end{lemma}

For the rest of this section, we let $P$ be set of patterns, such that for
all $\pi$ in $P$:
\begin{itemize}
    \item $\pi$ avoids \mptwodec{}, and
    \item $\pi \not\in \perms \oplus (\Avp(12)\backslash\{1\})$.
\end{itemize}
Note, $r_d$ and $2134$ satisfy these conditions.
As further examples, the permutations $312$ and $1423$ also satisfy the
conditions.

For two graphs on staircase grids of different sizes, we define the \emph{merge}
of those graphs by gluing them by the top right corner cell.
Figure~\ref{fig:merge_core} shows an example of a merge. The merge of two graphs
$A$ and $B$ is denoted $A\mergecore B$.

\begin{figure}[htpb]
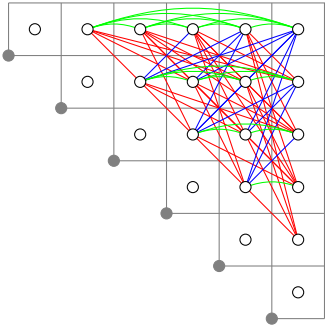

    \centering
    \includestandalone{figures/merge_core}
    \caption{
        The merge of $\UU(B_5)$, $\R(B_5)$ and $\D(B_6)$. The green edges
        come from $\R(B_5)$, the blue ones from $\UU(B_5)$, and the red ones from
        $\D(B_6)$.
    }%
    \label{fig:merge_core}
\end{figure}

In order to describe the structure of the staircase encodings
of the permutations in
$\Av(r_d, 2134, 1 \oplus P)$,
we define a labelling $\phi$ where the set of labels is
$\{y,z,s,t\}$. For a subset $I$ of the staircase grid and vertex $v$ of this
set, we let
\begin{equation*}
    \phi(v,I) =
    \begin{cases}
        y & \text{if $v$ is not in the leading diagonal},              \\
        z & \text{if there is a $v' \in I$ in the same column},        \\
        s & \text{if there is a $v' \in I$ that is north east of $v$}, \\
        t & \text{otherwise}.
    \end{cases}
\end{equation*}
Figure~\ref{fig:indep-DmUR} shows an independent set labelled with $\phi$.

\begin{figure}[htpb]
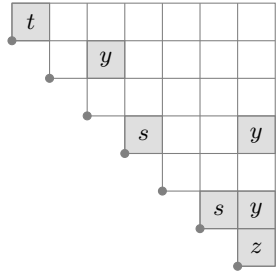

    \centering
    \includestandalone{figures/complex_indep_set}
    \caption{The labelling of an independent set with $\phi$.}%
    \label{fig:indep-DmUR}
\end{figure}

As we did for $\WI_{\rl}$, we define $\WI_\phi(G,Y,Z,S,T)$, as the set of
weighted independent sets of $G$ such that for a vertex $v$ in the independent set
$I$, the weights of $v$ is an element of
\begin{itemize}
    \item $Y$ if $\phi(v,I)=y$,
    \item $Z$ if $\phi(v,I)=z$,
    \item $S$ if $\phi(v,I)=s$, and
    \item $T$ if $\phi(v,I)=t$.
\end{itemize}

During the rest of the section we consider the set
\[
    \WI_\phi(\D(B_n)\mergecore\UR(B_{n-1}),\Avp(12),
    \Avp(r_d,2134,P)\backslash\{1\}, \Avp(213,\bstrip{P}),
    \Avp(r_d,2134,P)).
\]
For sake of brevity, we name it $\WIs_n$ in this section.

\begin{theorem}\label{thm:inf_rd_2134}
    There is a bijection between $\Avn(r_d, 2134, 1 \oplus P)$ and $\WIs_n$.
\end{theorem}

\begin{proof}
    Let $I$ be a weighted independent set in $\WIs_n$.
    Let $E$ be the staircase encoding such that cell $v$ contains the
    corresponding permutation, except when $\phi(v,I)=z$.
    In this case, we write the weight as $\alpha m \beta$ where $m$ is the
    maximum, and in this cell of $E$, we add $\alpha\beta$.

    Define $f$ to be a map which maps $I$ to the
    permutation $f(I)$ with staircase encoding $E$ such that
    \begin{itemize}
        \item the rows of $f(I)$ are increasing,
        \item excluding points in the leading diagonal, the columns are
              decreasing,
        \item in an active cell $v$ labelled $z$, with weight $\alpha m \beta$,
              $\alpha$ is to the left and $\beta$ is to the right of the points in
              the column.
    \end{itemize}
    Figure~\ref{fig:example-bij-rd-2134} shows the map $f$ applied to an
    independent set.
    \begin{figure}[htpb]
        \centering
        \includestandalone{figures/example_bij_rd_2134}
        \caption{The map $f$ from $\WIs_n$ to $\perms$.}%
        \label{fig:example-bij-rd-2134}
    \end{figure}

    We will show that $f$ is the bijection desired.
    Assume that $\sigma = f(I)$ contains $1 \oplus \pi$ for some $\pi \in P$.
    If $\sigma$ contains $1 \oplus \pi$, then it contains an occurrence where
    the $1$ in the occurrence is a left-to-right minimum in $\sigma$. Therefore,
    $\sigma$ contains an occurrence $\pi$ in a rectangular region of the
    staircase grid. This region is pictured in
    Figure~\ref{fig:rd-2134-avoid-pi}~(a) with the cell in the leading diagonal
    in blue.

    \begin{figure}[htpb]
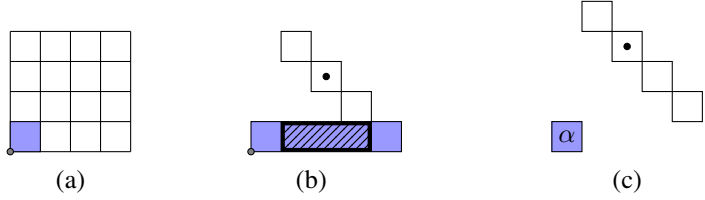

        \centering
        \includestandalone{figures/decomp_rd_2134}
        \caption{Decomposition of the rectangular region containing the
            occurrences of $\pi$.}%
        \label{fig:rd-2134-avoid-pi}
    \end{figure}

    In this region, without loss of generality we can assume the cell in the
    lower left corner is in the leading diagonal.
    The permutations that fill the active cells avoid $P$, therefore, the
    occurrence has points in at least two cells.
    In this region, consider the leftmost
    cell not in the leading diagonal containing a point of
    the occurrence of $\pi$. If this cell is in the first column, then by the
    definition of $I$ the columns to the right are not active in $f(I)$.
    The column consists of a decreasing permutation in the cells
    not in the leading diagonal (see Figure~\ref{fig:rd-2134-avoid-pi}~(b)).
    As $\pi$ avoids \mptwodec, the occurrence
    can use exactly one point in these cells, say $k$.
    Therefore, an occurrence of $\pi$ is of the form $\alpha k \beta$, where the
    $\alpha$ and $\beta$ are in the cell in leading diagonal, which contradicts
    the definition of $f$.

    Otherwise, the leftmost cell not in the leading diagonal is not in the
    first column. By the definition of $I$, the columns to the right and
    left are empty if they are not in the diagonal. Moreover, this column is
    decreasing as shown in Figure~\ref{fig:rd-2134-avoid-pi}~(c).
    Therefore, since the blue cell avoids $\bstrip{\pi}$,
    it follows that $\pi \in \perms \oplus (\Avp(12)\backslash\{1\})$,
    contradicting the second condition of $P$. Hence, we have shown that
    $f(I)$ avoids $1 \oplus \pi$.

    If $\sigma$ contains an occurrence of $2134$ then it either contains an
    occurrence of $1\oplus 2134$ or an occurrence of $(2134, \{(0,0)\})$.
    By the Shading Lemma, the latter implies an occurrence of $m_1$,
    see Figure~\ref{fig:mesh-2134-rd}.
    If $m_1$ occurs then the $2$ and the $1$ of
    the occurrence are left-to-right minima of the permutation. The $3$ and $4$
    of the occurrence violate either a decreasing cell constraint if they are
    in the same cell or an up-edge or decreasing column constraint
    if they are not.
    Hence, if $\sigma$ contains an occurrence of $2134$ it also contains
    an occurrence of $1\oplus 2134$. As $2134$ satisfies the conditions of $P$
    this implies $\sigma$ avoids $2134$.

    \begin{figure}[htpb]
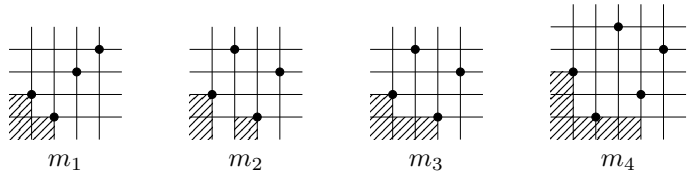

        \centering
        \includestandalone{figures/mp_2134_rd}
        \caption{Mesh patterns that do not occur in $\sigma$.}%
        \label{fig:mesh-2134-rd}
    \end{figure}

    If $\sigma$ contains an occurrence of $r_d$ then by a similar argument as
    above, it either contains an occurrence of $1\oplus r_d$ or $m_2$.
    Moreover, if
    $\sigma$ contains an occurrence of $m_2$, we can find an occurrence of
    $m_3$ or $m_4$. An occurrence of $m_3$ would violate the row increasing
    constraint or a down-edge constraint. The $5$ and the $2$ in an
    occurrence of $m_4$
    are in the same column of the staircase grid
    as they cannot have any left-to-right minima between them.
    The $5$ and $4$ in an occurrence of $m_4$ are also in the same column,
    otherwise they violate a down-edge or a row-edge constraint.
    Hence, since the $2$ in an occurrence is in a different
    row than the $5$ and $4$, we have a violation of the way we
    build the column. Therefore, $m_4$ does not occur in $\sigma$ and $\sigma$
    avoids $m_2$. We conclude that if $\sigma$ contains $r_d$ then it contains
    $1\oplus r_d$. As $r_d$ satisfies the conditions of $P$
    this implies $\sigma$ avoids $r_d$.

    The injectivity of $f$ follows from the uniqueness of the map.
    For surjectivity, we consider a permutation $\sigma$ in
    $\Avn(r_d, 2134, 1\oplus P)$.
    The rows of $\sigma$ are increasing by
    Lemma~\ref{lem:row_col_interleaving}.
    By Lemma~\ref{lem:const_2134}, the cells
    labelled $y$ in $\sigma$ contain decreasing permutations. Cells
    labelled $s$ avoid $213$ and $\bstrip{P}$ since there is a guaranteed
    point of the permutation to the north east of the points in those cells.
    The cells labelled $z$ and $t$ avoid $r_d$, $2134$ and $P$.
    Moreover, the active cells of $\sigma$ form an independent set of
    $\D(B_n)\mergecore \UR(B_{n-1})$ by Lemma~\ref{lem:const_2134}
    and~\ref{lem:const_2134_rd}.

    We study how cells in the same column interact.
    We consider a column of the staircase grid.
    By Lemma~\ref{lem:const_2134}, except for the bottommost cell, the
    column is decreasing and each cell contains a decreasing sequence.
    If there is
    a point in the bottom cell with index between two points of the decreasing
    sequence, then it creates an occurrence of $r_d$.
    Hence, the bottommost cell can only
    have points on both sides of the decreasing sequences.
    Figure~\ref{fig:interleaving-column} shows a typical column.
    In the bottommost
    cell, any point in the gray region would create a $r_d$ pattern.
    Moreover, the content of this cell cannot create one of the forbidden
    patterns with one of the points above. Hence, it can only split in a place
    where a new maximum could be added without creating a pattern in
    $\{r_d, 2134\}\cup P$. The content of this cell comes from a permutation
    $\alpha m\beta \in\Av(r_d, 2134, P)$ where $m$ is the maximum, $\alpha$ is
    placed on the left of the decreasing sequence and $\beta$ on the right.

    \begin{figure}[htpb]
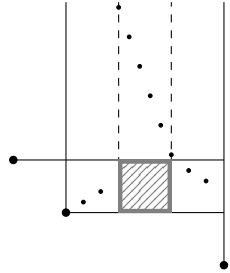

        \centering
        \includestandalone{figures/interleaving_column}
        \caption{A typical column for a permutation avoiding $2143$ and $r_d$.}%
        \label{fig:interleaving-column}
    \end{figure}

    This shows that $\sigma$ can be obtained from an element of $\WIs_n$ by
    applying $f$.
\end{proof}

To compute the generating function of $\Av(2134, r_d, 1\oplus P)$,
we need to compute the generating function for the independent set of
$\D(B_n)\mergecore\UR(B_{n-1})$ for $n\in\mathbb{N}$.
For these sets we track the number of vertices
with each label by a different variable.

\begin{proposition}\label{prop:gfDmUR}
    Let $\gfDmUR(x,y,z,s,t)$ be the generating function of independent sets of
    $\D(B_n)\mergecore\UR(B_{n-1})$ such that the variable $y,z,s,t$ track
    the number of
    vertices with labels $y,z,s,t$ in the set.
    Then $\gfDmUR(x,y,z,s,t)$ satisfies
    \[ \gfDmUR(x,y,z,s,t) = 1 + x(1+t)\gfDmUR(x,y,z,s,t) +
        \frac{x^2y(s+1)(z+1)}{1-x(s+1)(y+1)} \gfDmUR(x,y,z,s,t). \]
\end{proposition}
\begin{proof}
    We observe that the vertices in the leading diagonal of
    $\D(B_n)\mergecore\UR(B_{n-1})$ are disconnected from the graph. Hence, they can
    be freely added or removed from any independent set.

    Because of the row constraint on $B_{n-1}$, the topmost row can contain at
    most one vertex that is not in the leading diagonal.
    First, if the independent set does not contain such a vertex then it
    contributes
    $x(1+t)\gfDmUR(x,y,z,s,t)$
    to $\gfDmUR$.

    Otherwise, the graph decomposes as shown on
    Figure~\ref{fig:decomp_mergecore} and we get a contribution of
    \[ \frac{x^2y(s+1)(z+1)}{1-x(s+1)(y+1)} \gfDmUR(x,y,z,s,t). \]

    \begin{figure}[htpb]
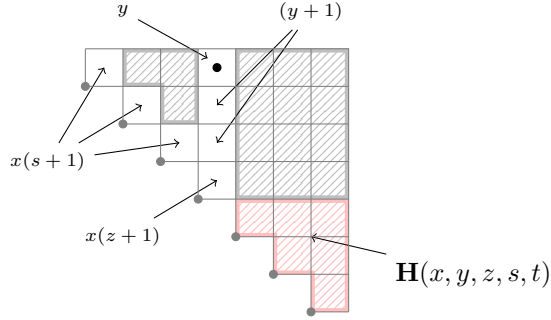

        \centering
        \includestandalone{figures/decomp_DmUR}
        \caption{
            The decomposition of an independent set of
            $\D(B_n)\mergecore\UR(B_{n-1})$ when a cell in
            the topmost row is active.
        }%
        \label{fig:decomp_mergecore}
    \end{figure}

    Hence, $\gfDmUR(x,y,z,s,t)$ satisfies
    \[ \gfDmUR(x,y,z,s,t) = 1 + x(1+t)\gfDmUR(x,y,z,s,t) +
        \frac{x^2y(s+1)(z+1)}{1-x(s+1)(y+1)} \gfDmUR(x,y,z,s,t)\]
    as claimed in the proposition.
\end{proof}

From Theorem~\ref{thm:inf_rd_2134} and Proposition~\ref{prop:gfDmUR}, we get the
enumeration of $\Av(r_d, 2134, 1\oplus P)$.
\begin{corollary}\label{cor:rd_2134}
    The generating function of $\Av(2134, 2413, 1\oplus P)$ is
    \[ \gfDmUR\left(x, \frac{x}{1-x}, \frac{B(x)-(1+x)}{x}, C(x)-1,
        B(x)-1\right)\]
    where
    \begin{itemize}
        \item $B(x)$ is the generating function of $\Av(2134, 2413, P)$,
        \item $C(x)$ is the generating function of $\Av(213, \bstrip{P})$.
    \end{itemize}
\end{corollary}

Using the corollary above, we can compute $A(x)$, the generating function of
$\Av(2134, 2413)$ that was first enumerated by \cite{MR3211768}.
In this example, $P$ is empty. Hence, $B(x)=A(x)$ and $C(x)$
is the generating function for the Catalan numbers. We get that the generating
function $A(x)$ satisfies
\[
    A(x) = \gfDmUR\left(x, \frac{x}{1-x}, \frac{A(x)-(1+x)}{x}, C(x)-1,
    A(x)-1\right).
\]
This equation can be solved explicitly to find the enumeration that appears in
OEIS as \oeis{A165538}.

%% file: figures/2134.tex
\begin{tikzpicture}[scale=.3,baseline=(current bounding box.center)]
	\foreach \x in {1,...,4} {
		\draw[ultra thin] (\x,0)--(\x,5); 
		\draw[ultra thin] (0,\x)--(5,\x); 
	}
	\draw[fill=black] (1,2) circle (5pt);
	\draw[fill=black] (2,1) circle (5pt);
	\draw[fill=red] (3,3) circle (5pt);
	\draw[fill=red] (4,4) circle (5pt);
\end{tikzpicture}

%% file: figures/indep_sub_graph.tex
\begin{tikzpicture}[yscale=-1, scale=0.5]
    \foreach \x in {1,...,6}{
        \draw[fill=gray, gray] (\x,\x) circle (3pt);
        \draw[gray] (\x,0)--(\x,\x)--(7,\x);
    }
    \draw[gray] (1,0)--(7,0)--(7,6);
    \foreach \y in {2,...,4}
    	\pgfmathtruncatemacro{\yu}{\y+1}
    	\pgfmathtruncatemacro{\yd}{\y-1}
    	\foreach \x in {\yu,...,5}
    		\pgfmathtruncatemacro{\xu}{\x+1}
    		\foreach \yy in {1,...,\yd}
    			\foreach \xx in {\xu,...,6}
        			\draw (\x+0.5,\y-0.5)--(\xx+0.5,\yy-0.5);
    		
    \foreach \y in {1,...,6} \foreach \x in {\y,...,6}{
        \draw[fill=white] (\x+0.5,\y-0.5) circle (3pt);
    }
    \foreach \x/\y in {1/1, 3/1, 3/2, 5/2, 5/3, 4/4, 6/5}{
        \draw[fill=black] (\x+0.5,\y-0.5) circle (3pt);
    }
\end{tikzpicture}

%% file: figures/merge_core.tex
\begin{tikzpicture}[yscale=-1, scale=0.7]
    \foreach \x in {1,...,6} {
        \draw[fill=gray, gray] (\x,\x) circle (3pt);
        \draw[gray] (\x,0)--(\x,\x)--(7,\x);
    }
    \draw[gray] (1,0)--(7,0)--(7,6);
    \foreach \y in {1,...,4}
	    \pgfmathtruncatemacro{\yu}{\y+1}
    	\foreach \x in {\yu,...,5}
    		\pgfmathtruncatemacro{\xu}{\x+1}
    		\foreach \yy in {\yu,...,\x} \foreach \xx in {\xu,...,6}
        		\draw[color=red] (\x+0.5,\y-0.5)--(\xx+0.5,\yy-0.5);
    \foreach \y in {2,...,4}
    	\pgfmathtruncatemacro{\yu}{\y+1}
    	\pgfmathtruncatemacro{\yd}{\y-1}
    	\foreach \x in {\yu,...,5}
    		\pgfmathtruncatemacro{\xu}{\x+1}
    		\foreach \yy in {1,...,\yd} \foreach \xx in {\xu,...,6}
        		\draw[color=blue] (\x+0.5,\y-0.5)--(\xx+0.5,\yy-0.5);
    \foreach \y in {1,...,4}
    	\pgfmathtruncatemacro{\yu}{\y+1}
    	\foreach \x in {\yu,...,5}
    		\pgfmathtruncatemacro{\xu}{\x+1}
    		\foreach \xx in {\xu,...,6}
        		\draw[color=green] (\x+0.5,\y-0.5) to[out=-20, in =-160] (\xx+0.5,\y-0.5);
    \foreach \y in {1,...,6} \foreach \x in {\y,...,6}
        \draw[fill=white] (\x+0.5,\y-0.5) circle (3pt);
\end{tikzpicture}

%% file: figures/complex_indep_set.tex
\begin{tikzpicture}[yscale=-1, scale=0.5]
    \newcommand{\drawVertex}[3]{
        \draw[gray, fill=gray!25] (#1,#2) rectangle+(1,-1);
        \node at ({#1+0.5},{#2-0.5}) {\small $#3$};
    }
    \drawVertex{1}{1}{t}
    \drawVertex{3}{2}{y}
    \drawVertex{4}{4}{s}
    \drawVertex{7}{4}{y}
    \drawVertex{7}{6}{y}
    \drawVertex{6}{6}{s}
    \drawVertex{7}{7}{z}
    \drawGrid{7}
\end{tikzpicture}

%% file: figures/example_bij_rd_2134.tex
\begin{tikzpicture}[yscale=-1, scale=0.6]
    \newcommand{\drawVertex}[3]{
        \node at ({#1+0.5},{#2-0.5}) {\scriptsize #3};
    }
    \drawVertex{1}{1}{$213$}
    \drawVertex{2}{2}{$123$}
    \drawVertex{4}{2}{$21$}
    \drawVertex{4}{3}{$321$}
        \drawVertex{4}{4}{$\textcolor{blue}{2143}$}
    \drawGrid{4}
    \draw[->] (5.5, 2)--(7.5,2);
    \node[above] at (6.5, 2) {$f$};

    \begin{scope}[shift={(8,0)}]
        \draw[gray!50] (0,0)--(7,0)--(7,5);
        \draw[gray!50] (0,0)--(0,1)--(7,1);
        \draw[gray!50] (1,0)--(1,3)--(7,3);
        \draw[gray!50] (2,0)--(2,4)--(7,4);
        \draw[gray!50] (3,0)--(3,5)--(7,5);
        \draw[black,fill=black] (0,1) circle (2pt);
        \draw[black,fill=black] (1,3) circle (2pt);
        \draw[black,fill=black] (2,4) circle (2pt);
        \draw[black,fill=black] (3,5) circle (2pt);
        \draw[black,fill=black] (0,1) circle (2pt);

        \draw[gray!50, dashed] (1,2)--(7,2);
        \draw[gray!50, dashed] (4,0)--(4,5);
        \draw[gray!50, dashed] (5,0)--(5,5);
        \draw[gray!50, dashed] (6,0)--(6,5);

        \draw[fill=black] (0.25,0.50) circle (2pt);
        \draw[fill=black] (0.50,0.75) circle (2pt);
        \draw[fill=black] (0.75,0.25) circle (2pt);

        \draw[fill=black] (1.25,2.75) circle (2pt);
        \draw[fill=black] (1.50,2.50) circle (2pt);
        \draw[fill=black] (1.75,2.25) circle (2pt);

        \draw[fill=black] (4.33,1.33) circle (2pt);
        \draw[fill=black] (4.66,1.66) circle (2pt);

        \draw[fill=black] (5.25,3.25) circle (2pt);
        \draw[fill=black] (5.50,3.50) circle (2pt);
        \draw[fill=black] (5.75,3.75) circle (2pt);

        \draw[blue, fill=blue] (3.33,4.50) circle (2pt);
        \draw[blue, fill=blue] (3.66,4.75) circle (2pt);
        \draw[blue, fill=blue] (6.50,4.25) circle (2pt);

        \node at (4,6) {\small$(15)(17)(16)(18)9(10)(11)(12)5132(14)(13)8764$};
    \end{scope}
\end{tikzpicture}

%% file: figures/decomp_rd_2134.tex
\begin{tikzpicture}[yscale=1, scale=0.4]
    \foreach\x in {0,...,4}{
        \draw (\x,0)--(\x,4);
        \draw (0,\x)--(4,\x);
    }
    \draw[fill=blue!40] (0, 0) rectangle +(1,1);
    \draw[fill=gray] (0,0) circle (3pt);
    \node at (2,-1) {(a)};

    \begin{scope}[shift={(8,0)}]
        \draw[fill=blue!40] (0,0) rectangle +(5,1);
        \highlightRegion{black}{(1,0)--(1,1)--(4,1)--(4,0)}
        \draw (1, 3) rectangle +(1,1);
        \draw (2, 2) rectangle +(1,1);
        \draw (3, 1) rectangle +(1,1);
        \draw (1,0)--(1,1);
        \draw (4,0)--(4,1);
        \draw[fill=black] (2.5,2.5) circle (3pt);
        \draw[fill=gray] (0,0) circle (3pt);
        \node at (2,-1) {(b)};
    \end{scope}

    \begin{scope}[shift={(18,0)}]
        \draw[fill=blue!40] (0, 0) rectangle +(1,1);
        \node at (0.5,0.5) {$\alpha$};
        \draw[] (1, 4) rectangle +(1,1);
        \draw[] (2, 3) rectangle +(1,1);
        \draw[] (3, 2) rectangle +(1,1);
        \draw[] (4, 1) rectangle +(1,1);
        \draw[fill=black] (2.5,3.5) circle (3pt);

        \node at (2.5,-1) {(c)};
    \end{scope}
\end{tikzpicture}

%% file: figures/mp_2134_rd.tex
\begin{tikzpicture}[scale=.3,baseline=(current bounding box.center)]
    \foreach \x in {1,...,4} {
        \draw[ultra thin] (\x,0)--(\x,5); 
        \draw[ultra thin] (0,\x)--(5,\x); 

    }
    \fill[pattern=north east lines] (0, 1) rectangle +(1,1);
    \fill[pattern=north east lines] (0, 0) rectangle +(1,1);
    \fill[pattern=north east lines] (1, 0) rectangle +(1,1);
    \draw[fill=black] (1,2) circle (5pt);
    \draw[fill=black] (2,1) circle (5pt);
    \draw[fill=black] (3,3) circle (5pt);
    \draw[fill=black] (4,4) circle (5pt);
    \node at (2.5,-1) {$m_1$};

    \begin{scope}[shift={(8,0)}]
        \foreach \x in {1,...,4} {
            \draw[ultra thin] (\x,0)--(\x,5); 
            \draw[ultra thin] (0,\x)--(5,\x); 

        }
        \fill[pattern=north east lines] (0, 1) rectangle +(1,1);
        \fill[pattern=north east lines] (0, 0) rectangle +(1,1);
        \fill[pattern=north east lines] (2, 0) rectangle +(1,1);
        \draw[fill=black] (1,2) circle (5pt);
        \draw[fill=black] (2,4) circle (5pt);
        \draw[fill=black] (3,1) circle (5pt);
        \draw[fill=black] (4,3) circle (5pt);
        \node at (2.5,-1) {$m_2$};
    \end{scope}
    \begin{scope}[shift={(16,0)}]
        \foreach \x in {1,...,4} {
            \draw[ultra thin] (\x,0)--(\x,5); 
            \draw[ultra thin] (0,\x)--(5,\x); 

        }
        \fill[pattern=north east lines] (0, 1) rectangle +(1,1);
        \fill[pattern=north east lines] (0, 0) rectangle +(1,1);
        \fill[pattern=north east lines] (1, 0) rectangle +(1,1);
        \fill[pattern=north east lines] (2, 0) rectangle +(1,1);
        \draw[fill=black] (1,2) circle (5pt);
        \draw[fill=black] (2,4) circle (5pt);
        \draw[fill=black] (3,1) circle (5pt);
        \draw[fill=black] (4,3) circle (5pt);
        \node at (2.5,-1) {$m_3$};
    \end{scope}

    \begin{scope}[shift={(24,0)}]
        \foreach \x in {1,...,5} {
            \draw[ultra thin] (\x,0)--(\x,6); 
            \draw[ultra thin] (0,\x)--(6,\x); 

        }
        \fill[pattern=north east lines] (0, 2) rectangle +(1,1);
        \fill[pattern=north east lines] (0, 1) rectangle +(1,1);
        \fill[pattern=north east lines] (0, 0) rectangle +(1,1);
        \fill[pattern=north east lines] (1, 0) rectangle +(1,1);
        \fill[pattern=north east lines] (2, 0) rectangle +(1,1);
        \fill[pattern=north east lines] (3, 0) rectangle +(1,1);
        \draw[fill=black] (1,3) circle (5pt);
        \draw[fill=black] (2,1) circle (5pt);
        \draw[fill=black] (3,5) circle (5pt);
        \draw[fill=black] (4,2) circle (5pt);
        \draw[fill=black] (5,4) circle (5pt);
        \node at (3,-1) {$m_4$};

    \end{scope}
\end{tikzpicture}

%% file: figures/interleaving_column.tex
\begin{tikzpicture}[scale=0.7]
    \draw (0,4)--(0,0)--(3,0);
    \draw (-1,1)--(3,1);
    \draw (3,-1)--(3,4);
    \draw[dashed] (1,0)--(1,4);
    \draw[dashed] (2,0)--(2,4);
    \draw[fill=black] (0,0) circle (2pt);
    \draw[fill=black] (-1,1) circle (2pt);
    \draw[fill=black] (3,-1) circle (2pt);
    \highlightRegion{gray}{(1,0)--(1,1)--(2,1)--(2,0)}

    \foreach \x in {10, 12, 14, 16, 18, 20} {
        \draw[fill=black] (0.1*\x,-2.8*0.1*\x+6.7) circle (1pt);
    }
    \draw[fill=black] (0.33, 0.2) circle (1pt);
    \draw[fill=black] (0.66, 0.4) circle (1pt);
    \draw[fill=black] (2.66, 0.6) circle (1pt);
    \draw[fill=black] (2.33, 0.8) circle (1pt);
\end{tikzpicture}

%% file: figures/decomp_DmUR.tex
\begin{tikzpicture}[yscale=-1, scale=0.5]
    \draw[black, fill=black] (4.5, 0.5) circle (3pt);
    \highlightRegion{gray!50}{(2,0)--(2,1)--(3,1)--(3,2)--(4,2)--(4,0)}
    \highlightRegion{gray!50}{(5,0)--(8,0)--(8,4)--(5,4)}
    \highlightRegion{pink}{(5,4)--(5,5)--(6,5)--(6,6)--(7,6)--(7,7)--(8,7)--(8,4)}

    \node at (0,3) (s) {\scriptsize$x(s+1)$};
    \draw[->] (s)--(1.5,0.5);
    \draw[->] (s)--(2.5,1.5);
    \draw[->] (s)--(3.5,2.5);
    \node at (7,-1) (y1) {\scriptsize$(y+1)$};
    \draw[->] (y1)--(4.5,1.5);
    \draw[->] (y1)--(4.5,2.5);
    \node at (2,-1) (y) {\scriptsize$y$};
    \draw[->] (y)--(4.25,0.25);
    \node at (2,5) (z) {\scriptsize$x(z+1)$};
    \draw[->] (z)--(4.5,3.5);
    \node[right] at (9,6) (H) {$\gfDmUR(x,y,z,s,t)$};
    \draw[->] (H)--(7,5);
    \drawGrid{7}
\end{tikzpicture}

%% file: 10-the-other-crazy.tex
\section{Avoiding \texorpdfstring{$r_u$}{ru} and
  \texorpdfstring{$2143$}{2143}}\label{sec:ru_2143}

Using similar techniques as in the previous section we enumerate classes of the
form $\Av(r_u, 2143, 1\oplus P)$ where each pattern $\pi$ in $P$ satisfies
\begin{itemize}
    \item $\pi$ avoids \mptwoinc, and
    \item $\pi \not\in \perms \ominus \Avp(21)$.
\end{itemize}
For the entire section, we let $P$ be such a set.

We look at the pattern $2143$ in Figure~\ref{fig:2143} as we did for $2134$.
If we consider the two black points as left-to-right minima of a permutation,
then the two red points will enforce a down-core structure on the staircase
grid except for the leading diagonal.
\begin{figure}[htpb]
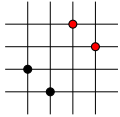

    \centering
    \includestandalone{figures/2143}
    \caption{The pattern 2143.}%
    \label{fig:2143}
\end{figure}
The following lemmas follow from similar arguments as in
Lemmas~\ref{lem:row_col_interleaving},~\ref{lem:up_down_edge_const}
and~\ref{lem:col_row_edges_const}.
\begin{lemma}\label{lem:const_2143}
    Let $\sigma$ be a permutation in $\Avn(2143)$.
    We consider the set $C$ of active cells of the staircase encoding of $\sigma$
    that are not in the main diagonal of the grid. Then
    \begin{itemize}
        \item $C$ is an independent set of $D(B_{n-1})$,
        \item cells in $C$ contain increasing permutations.
    \end{itemize}
    Moreover, if we remove the cells of the leading diagonal from $\sigma$,
    the rows and columns are increasing.
\end{lemma}

\begin{lemma}\label{lem:const_2143_ru}
    Let $\sigma$ be a permutation in $\Avn(2143, r_u)$.
    The set of active cells of the staircase encoding of $\sigma$ is
    an independent set of $\R(B_{n-1})$.
\end{lemma}
From the two previous lemmas and Lemma~\ref{lem:up_down_edge_const},
we know that for $\sigma$ in
$\Avn(r_u,2143)$ the active cells of $\SE(\sigma)$ are an independent set of
$\UU(B_n)\mergecore\DR(B_{n-1})$.
To describe the structure of the staircase encoding of the permutations in
$\Av(r_u, 2143, 1 \oplus P)$ we introduce new labelled weighted independent sets.
First, for a subset $I$ of the staircase grid and a vertex $v$ of $I$, we set
\begin{equation*}
    \psi(v,I) =
    \begin{cases}
        y & \text{if $v$ is not in the leading diagonal},       \\
        z & \text{if there is a $v' \in I$ in the same column}, \\
        s & \text{otherwise}.
    \end{cases}
\end{equation*}
Figure~\ref{fig:indep-UmDR} shows of a independent set labelled with $\psi$.
\begin{figure}[htpb]
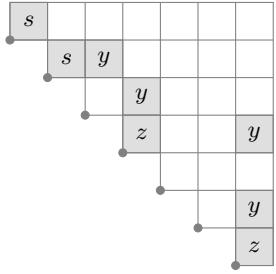

    \centering
    \includestandalone{figures/complex_indep_set2}
    \caption{The labelling with $\psi$ of an independent set.}%
    \label{fig:indep-UmDR}
\end{figure}

We define $\WI_\psi(G,Y,Z,S)$ as the set of
weighted independent sets of $G$ such that for a vertex $v$ in the independent set
$I$ the weight of $v$ is an element of
\begin{itemize}
    \item $Y$ if $\psi(v,I)=y$,
    \item $Z$ if $\psi(v,I)=z$, and
    \item $S$ if $\psi(v,I)=s$.
\end{itemize}

\begin{theorem}\label{thm:inf_ru_2143}
    There is a bijection between $\Avn(r_u, 2143, 1 \oplus P)$ and
    \[
        \WI_\psi(\UU(B_n)\mergecore\DR(B_{n-1}), \Avp(21),
        \Avp(r_u,2143,P)\backslash\{1\}, \Avp(r_u, 2143,P)).
    \]
\end{theorem}

\begin{proof}
    For conciseness, we denote with $\WIs_n$ the set of weighted independent
    sets stated in the theorem.
    Let $I$ be a weighted independent set in $\WIs_n$.

    Let $E$ be the staircase encoding such that a cell $v$ contains the same
    permutation as in $I$, except when $\phi(v,I)=z$.
    In this case, we write the weight as $\alpha m \beta$ where $m$ is the
    maximum, and add $\alpha\beta$ to cell $v$.

    Define $f$ to be a map which maps $I$ to the
    permutation $f(I)$ with staircase encoding $E$ such that
    \begin{itemize}
        \item the rows of $f(I)$ are decreasing,
        \item excluding points in the leading diagonal, the columns are increasing,
        \item in an active cell $v$ labelled $z$, with weight $\alpha m \beta$,
              $\alpha$ is to the left and $\beta$ is to the right of the points in
              the column.
    \end{itemize}
    Figure~\ref{fig:example-bij-ru-2143} shows the map $f$ applied to an
    independent set.
    \begin{figure}[htpb]
        \centering
        \includestandalone{figures/example_bij_ru_2143}
        \caption{The map $f$ from $\WIs_n$ to $\perms$.}%
        \label{fig:example-bij-ru-2143}
    \end{figure}

    We show that $f$ is the desired bijection.
    Assume that $\sigma = f(I)$ contains $1 \oplus \pi$ for some $\pi \in P$.
    If $\sigma$ contains $1 \oplus \pi$, then it contains an occurrence where
    the $1$ in the occurrence is a left-to-right minimum in $\sigma$. Therefore,
    $\sigma$ contains an occurrence $\pi$ in a rectangular region of the
    staircase grid. This region is pictured in
    Figure~\ref{fig:ru-2143-avoid-pi}~(a) with the cell in the leading diagonal
    colored blue.

    \begin{figure}[htpb]
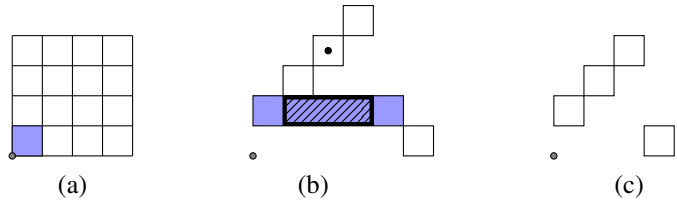

        \centering
        \includestandalone{figures/decomp_ru_2143}
        \caption{Decomposition of the rectangular region containing the
            occurrences of $\pi$.}%
        \label{fig:ru-2143-avoid-pi}
    \end{figure}

    In this region, without loss of generality, we can assume the cell in the
    lower left corner is in the leading diagonal.
    The permutations contained in each
    active cell avoid $P$, therefore, the
    occurrence has points in at least two cells.
    In this region, consider the leftmost
    cell not in the leading diagonal containing a point of
    the occurrence of $\pi$. If this cell is in the first column then, by the
    definition of $I$, only cells in the leftmost column and the bottommost row
    can be active. Moreover, the row-edges imply that only one cell in the
    bottom row that is not in the leading diagonal can be active.
    The column consists of an increasing permutation in the cells
    not in the leading diagonal and the rows are decreasing
    (see Figure~\ref{fig:ru-2143-avoid-pi}~(b)).
    As $\pi$ is not in $\perms\ominus \Avp(21)$, the bottommost cell in this
    figure is empty.
    Also, as $\pi$ avoids \mptwoinc, the occurrence
    can use exactly one point in the remaining white cells, say $k$.
    Therefore, the occurrence of $\pi$ is of the form $\alpha k \beta$, where
    the $\alpha$ and $\beta$ are in the cell in leading diagonal,
    which contradicts the definition of $f$.

    Otherwise, the leftmost cell not in the leading diagonal is not in the
    first column. By the definition of $I$, only the column of this cell and the
    bottommost row can be active.
    Moreover, this column is
    increasing and the bottom row is decreasing.
    Only one white cell of the bottom row can be active because of the
    row-edges.
    This is shown in Figure~\ref{fig:ru-2143-avoid-pi}~(c).
    Therefore, it follows, since $\pi \not\in \perms \ominus \Avp(21)$, that
    the bottommost cell is empty. Hence $\pi$ is an increasing
    permutation, which contradicts the first condition on $P$.
    Hence, we have shown that $f(I)$ avoids $1 \oplus \pi$.

    If $\sigma$ contains an occurrence of $2143$ then it either contains an
    occurrence of $1\oplus 2143$ or an occurrence of $(2143, \{(0,0)\})$.
    By the Shading Lemma, the latter implies an occurrence of $m_1$
    (see Figure~\ref{fig:mesh-2143-ru}).
    If $m_1$ occurs then the $2$ and the $1$ of
    the occurrence are left-to-right minima of the permutation. The $3$ and $4$
    of the occurrence violate either an increasing cell constraint if they are
    in the same cell or a down-edge constraint if they are not.
    Hence, if $\sigma$ contains an occurrence of $2143$ it also contains
    an occurrence of $1\oplus 2143$. As $2143$ satisfies the conditions of $P$
    this implies $\sigma$ avoids $2143$.

    \begin{figure}[htpb]
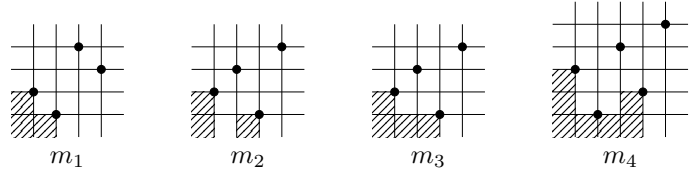

        \centering
        \includestandalone{figures/mp_2143_ru}
        \caption{Mesh patterns that do not occur in $\sigma$.}%
        \label{fig:mesh-2143-ru}
    \end{figure}

    If $\sigma$ contains an occurrence of $r_u$ then by a similar argument as
    above, it either contains an occurrence of $1\oplus r_u$ or $m_2$.
    Moreover, if
    $\sigma$ contains an occurrence of $m_2$, we can find an occurrence of
    $m_3$ or $m_4$. An occurrence of $m_3$ would violate the row decreasing
    constraints or the up-edges constraints. The $4$ and the $2$ in an
    occurrence of $m_4$ are in the same column of the staircase grid
    since they cannot have any left-to-right minima between them.
    The $5$ and $4$ in an occurrence of $m_4$ are also in the same column
    otherwise they violate an up-edge constraint.
    Hence, since the point at index $4$ is in a different
    row than the ones at index $3$ and $5$, we have a violation of the way we
    build the column. Therefore, $m_4$ does not occur in $\sigma$ and $\sigma$
    avoids $m_2$. We conclude that if $\sigma$  contains $r_u$ then it contains
    $1\oplus r_u$. As $r_u$ satisfies the conditions of $P$,
    this implies $\sigma$ avoids $r_u$.

    The injectivity of $f$ follows from the uniqueness of the map.
    For surjectivity, we consider a permutation $\sigma$ in
    $\Avn(r_u, 2143, 1\oplus P)$.
    The rows of $\sigma$ are decreasing by
    Lemma~\ref{lem:row_col_interleaving}.
    By Lemma~\ref{lem:const_2143}, the cells
    labelled $y$ in $\sigma$ contain increasing permutations.
    The cells labelled $z$ and $s$ avoid $r_u$, $2143$ and $P$.
    Moreover, the active cells of $\sigma$ form an independent set of
    $\UU(B_n)\mergecore \DR(B_{n-1})$ by Lemmas~\ref{lem:const_2143}
    and~\ref{lem:const_2143_ru}.

    Finally, we need to consider the points in a column of the staircase grid.
    By Lemma~\ref{lem:const_2143}, except for the bottommost cell, the
    column is increasing and each cell contains an increasing permutation.
    If there is
    a point in the bottom cell with index between two points of the increasing
    sequence, then it creates an occurrence of $r_u$.
    Hence, the bottommost cell can only
    have points on either side of the increasing sequence.
    Figure~\ref{fig:interleaving-column2} shows a typical column.
    \begin{figure}[htpb]
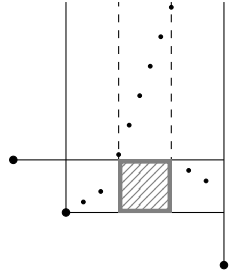

        \centering
        \includestandalone{figures/interleaving_column2}
        \caption{A typical column for a permutation avoiding $2143$ and $r_u$.}%
        \label{fig:interleaving-column2}
    \end{figure}
    In the bottommost
    cell, any point in the gray region would create an $r_u$ pattern.
    Moreover, the points in this cell together with another point in this column
    cannot create one of the forbidden patterns. Hence, it can only be split
    in a place where a new maximum could be added without creating a pattern in
    $\{r_u, 2143\}\cup P$. Therefore, the points in the cell are of the form
    $\alpha \beta$ where $\alpha$ is to the left of the other points in the
    column, $\beta$ is to the right of the other points in the column, and
    $\alpha m \beta \in\Av(r_u, 2143, P)$.

    Therefore we have shown that $\sigma$ can be obtained from an element of
    $\WIs_n$ by applying $f$.
\end{proof}

To compute the generating function of $\Av(r_u, 2143, 1\oplus P)$,
we need to compute the generating function for the independent set of
$\UU(B_n)\mergecore\DR(B_{n-1})$ for $n\in\mathbb{N}$.
For these sets we track the number of vertices
and their labels with different variables.

\begin{proposition}\label{prop:gfUmDR}
    Let $\gfUmDR(x,y,z,s)$ be the generating function of independent sets of
    $\D(B_n)\mergecore\UR(B_{n-1})$ such that the variables $y,z,s$
    track the number of vertices with label $y,z,s$ in the set.
    Then $\gfUmDR(x,y,z,s)$ satisfies
    \[ \gfUmDR(x,y,z,s) = 1 + x(s+1)\gfUmDR(x,y,z,s) +
        \frac{xy(z+1)(\gfUmDR(x,y,z,s)-1)}{1-x(y+1)}. \]
\end{proposition}
\begin{proof}
    We consider the top row of the graph. The leftmost cell is completely
    disconnected from the graph, hence, it can be freely added to the
    independent set.
    The row constraint on $B_{n-1}$ implies the topmost row contains at
    most one vertex that is not in the leading diagonal.
    First, if the independent set does not contain such a vertex then it
    contributes $x(s+1)\gfUmDR(x,y,z,s)$ to $\gfUmDR$.

    \begin{figure}[htpb]
        \centering
        \includestandalone{figures/decomp_UmDR}
        \caption{
            The decomposition of an independent set of
            $\UU(B_n)\mergecore\DR(B_{n-1})$ when a cell in
            the topmost row is active.
        }%
        \label{fig:decomp_UmDR}
    \end{figure}

    Otherwise, the graph decomposes as shown on
    Figure~\ref{fig:decomp_UmDR} and we get a contribution of
    \[
        \frac{x(s+1)y}{1-x(y+1)} x(z+1)\frac{\gfUmDR(x,y,z,s)-1}{x(s+1)}
    \]
    where the $\frac{\gfUmDR(x,y,z,s)}{x(s+1)}$ term is due to the pink region
    which is precisely a smaller core, where a single disconnected cell
    has been removed.
    Hence, $\gfUmDR(x,y,z,s,t)$ satisfies
    \[
        \gfUmDR(x,y,z,s) = 1 + x(s+1)\gfUmDR(x,y,z,s) +
        \frac{xy(s+1)}{1-x(y+1)}x(z+1)\frac{\gfUmDR(x,y,z,s)-1}{x(s+1)}
    \]
    as claimed in the proposition.
\end{proof}

From Theorem~\ref{thm:inf_ru_2143} and Proposition~\ref{prop:gfUmDR}, we get the
enumeration of $\Av(r_u, 2143, 1\oplus P)$.
\begin{corollary}\label{cor:ru_2143}
    The generating function of $\Av(2134, 2413, 1\oplus P)$ is
    \[ \gfUmDR\left(x, \frac{x}{1-x},\frac{B(x)-(1+x)}{x}, B(x)-1\right)\]
    where $B(x)$ is the generating function of $\Av(2314, 2143, P)$.
\end{corollary}

Using the corollary above, we can compute $A(x)$, the generating function of
$\Av(2314, 2143)$ that was first enumerated by \cite{MR2156679}.
In this example, $P$ is empty. Hence, $B(x)=A(x)$.
We get that the generating
function $A(x)$ satisfies
\[ A(x) = \gfUmDR\left(x, \frac{x}{1-x}, \frac{A(x)-(1+x)}{x}, A(x)-1\right).\]
This equation can be solved explicitly to find the generating function
\[\frac{1-\sqrt{1 -8x + 16x^2 -8x^3}}{4(x - x^2)}.\]
This generating function gives the enumeration that appears in
OEIS as \oeis{A109033}.

%% file: figures/2143.tex
\begin{tikzpicture}[scale=.3,baseline=(current bounding box.center)]
	\foreach \x in {1,...,4} {
		\draw[ultra thin] (\x,0)--(\x,5); 
		\draw[ultra thin] (0,\x)--(5,\x); 
	}
	\draw[fill=black] (1,2) circle (5pt);
	\draw[fill=black] (2,1) circle (5pt);
	\draw[fill=red] (3,4) circle (5pt);
	\draw[fill=red] (4,3) circle (5pt);
\end{tikzpicture}

%% file: figures/complex_indep_set2.tex
\begin{tikzpicture}[yscale=-1, scale=0.5]
    \newcommand{\drawVertex}[3]{
        \draw[gray, fill=gray!25] (#1,#2) rectangle+(1,-1);
        \node at ({#1+0.5},{#2-0.5}) {\small $#3$};
    }
    \drawVertex{1}{1}{s}
    \drawVertex{2}{2}{s}
    \drawVertex{3}{2}{y}
    \drawVertex{4}{3}{y}
    \drawVertex{4}{4}{z}
    \drawVertex{7}{4}{y}
    \drawVertex{7}{6}{y}
    \drawVertex{7}{7}{z}
    \drawGrid{7}
\end{tikzpicture}

%% file: figures/example_bij_ru_2143.tex
\begin{tikzpicture}[yscale=-1, scale=0.6]
    \newcommand{\drawVertex}[3]{
        \node at ({#1+0.5},{#2-0.5}) {\scriptsize $#3$};
    }
    \drawVertex{1}{1}{213}
    \drawVertex{2}{2}{123}
    \drawVertex{4}{2}{12}
    \drawVertex{4}{3}{123}
    \drawVertex{4}{4}{\textcolor{blue}{2431}}
    \drawGrid{4}
    \draw[->] (5.5, 2)--(7.5,2);
    \node[above] at (6.5, 2) {$f$};

    \begin{scope}[shift={(8,0)}]
        \draw[gray!50] (0,0)--(7,0)--(7,5);
        \draw[gray!50] (0,0)--(0,1)--(7,1);
        \draw[gray!50] (1,0)--(1,3)--(7,3);
        \draw[gray!50] (2,0)--(2,4)--(7,4);
        \draw[gray!50] (3,0)--(3,5)--(7,5);
        \draw[black,fill=black] (0,1) circle (2pt);
        \draw[black,fill=black] (1,3) circle (2pt);
        \draw[black,fill=black] (2,4) circle (2pt);
        \draw[black,fill=black] (3,5) circle (2pt);
        \draw[black,fill=black] (0,1) circle (2pt);

        \draw[gray!50, dashed] (1,2)--(7,2);
        \draw[gray!50, dashed] (4,0)--(4,5);
        \draw[gray!50, dashed] (5,0)--(5,5);
        \draw[gray!50, dashed] (6,0)--(6,5);

        \draw[fill=black] (0.25,0.50) circle (2pt);
        \draw[fill=black] (0.50,0.75) circle (2pt);
        \draw[fill=black] (0.75,0.25) circle (2pt);

        \draw[fill=black] (1.25,1.75) circle (2pt);
        \draw[fill=black] (1.50,1.50) circle (2pt);
        \draw[fill=black] (1.75,1.25) circle (2pt);

        \draw[fill=black] (5.33,2.66) circle (2pt);
        \draw[fill=black] (5.66,2.33) circle (2pt);

        \draw[fill=black] (4.25,3.75) circle (2pt);
        \draw[fill=black] (4.50,3.50) circle (2pt);
        \draw[fill=black] (4.75,3.25) circle (2pt);

        \draw[blue, fill=blue] (3.50,4.50) circle (2pt);
        \draw[blue, fill=blue] (6.33,4.25) circle (2pt);
        \draw[blue, fill=blue] (6.66,4.75) circle (2pt);

        \node at (4,6) {\small$(15)(17)(16)(18)9(12)(13)(14)513678(10)(11)42$};
    \end{scope}
\end{tikzpicture}

%% file: figures/decomp_ru_2143.tex
\begin{tikzpicture}[yscale=1, scale=0.4]
    \foreach\x in {0,...,4}{
        \draw (\x,0)--(\x,4);
        \draw (0,\x)--(4,\x);
    }
    \draw[fill=blue!40] (0, 0) rectangle +(1,1);
    \draw[fill=gray] (0,0) circle (3pt);
    \node at (2,-1) {(a)};

    \begin{scope}[shift={(8,0)}]
        \draw[fill=blue!40] (0,1) rectangle +(5,1);
        \highlightRegion{black}{(1,1)--(1,2)--(4,2)--(4,1)}
        \draw (1, 2) rectangle +(1,1);
        \draw (2, 3) rectangle +(1,1);
        \draw (3, 4) rectangle +(1,1);
        \draw (5, 0) rectangle +(1,1);
        \draw (1,1)--(1,2);
        \draw (4,1)--(4,2);
        \draw[fill=black] (2.5,3.5) circle (3pt);
        \draw[fill=gray] (0,0) circle (3pt);
        \node at (2,-1) {(b)};
    \end{scope}

    \begin{scope}[shift={(18,0)}]
        \draw[] (0, 1) rectangle +(1,1);
        \draw[] (1, 2) rectangle +(1,1);
        \draw[] (2, 3) rectangle +(1,1);
        \draw[] (3, 0) rectangle +(1,1);
        \draw[fill=gray] (0,0) circle (3pt);

        \node at (2.5,-1) {(c)};
    \end{scope}
\end{tikzpicture}

%% file: figures/mp_2143_ru.tex
\begin{tikzpicture}[scale=.3,baseline=(current bounding box.center)]
    \foreach \x in {1,...,4} {
        \draw[ultra thin] (\x,0)--(\x,5); 
        \draw[ultra thin] (0,\x)--(5,\x); 

    }
    \fill[pattern=north east lines] (0, 1) rectangle +(1,1);
    \fill[pattern=north east lines] (0, 0) rectangle +(1,1);
    \fill[pattern=north east lines] (1, 0) rectangle +(1,1);
    \draw[fill=black] (1,2) circle (5pt);
    \draw[fill=black] (2,1) circle (5pt);
    \draw[fill=black] (3,4) circle (5pt);
    \draw[fill=black] (4,3) circle (5pt);
    \node at (2.5,-1) {$m_1$};

    \begin{scope}[shift={(8,0)}]
        \foreach \x in {1,...,4} {
            \draw[ultra thin] (\x,0)--(\x,5); 
            \draw[ultra thin] (0,\x)--(5,\x); 

        }
        \fill[pattern=north east lines] (0, 1) rectangle +(1,1);
        \fill[pattern=north east lines] (0, 0) rectangle +(1,1);
        \fill[pattern=north east lines] (2, 0) rectangle +(1,1);
        \draw[fill=black] (1,2) circle (5pt);
        \draw[fill=black] (2,3) circle (5pt);
        \draw[fill=black] (3,1) circle (5pt);
        \draw[fill=black] (4,4) circle (5pt);
        \node at (2.5,-1) {$m_2$};
    \end{scope}
    \begin{scope}[shift={(16,0)}]
        \foreach \x in {1,...,4} {
            \draw[ultra thin] (\x,0)--(\x,5); 
            \draw[ultra thin] (0,\x)--(5,\x); 

        }
        \fill[pattern=north east lines] (0, 1) rectangle +(1,1);
        \fill[pattern=north east lines] (0, 0) rectangle +(1,1);
        \fill[pattern=north east lines] (1, 0) rectangle +(1,1);
        \fill[pattern=north east lines] (2, 0) rectangle +(1,1);
        \draw[fill=black] (1,2) circle (5pt);
        \draw[fill=black] (2,3) circle (5pt);
        \draw[fill=black] (3,1) circle (5pt);
        \draw[fill=black] (4,4) circle (5pt);
        \node at (2.5,-1) {$m_3$};
    \end{scope}

    \begin{scope}[shift={(24,0)}]
        \foreach \x in {1,...,5} {
            \draw[ultra thin] (\x,0)--(\x,6); 
            \draw[ultra thin] (0,\x)--(6,\x); 

        }
        \fill[pattern=north east lines] (0, 2) rectangle +(1,1);
        \fill[pattern=north east lines] (0, 1) rectangle +(1,1);
        \fill[pattern=north east lines] (0, 0) rectangle +(1,1);
        \fill[pattern=north east lines] (1, 0) rectangle +(1,1);
        \fill[pattern=north east lines] (2, 0) rectangle +(1,1);
        \fill[pattern=north east lines] (3, 0) rectangle +(1,1);
        \fill[pattern=north east lines] (3, 1) rectangle +(1,1);
        \draw[fill=black] (1,3) circle (5pt);
        \draw[fill=black] (2,1) circle (5pt);
        \draw[fill=black] (3,4) circle (5pt);
        \draw[fill=black] (4,2) circle (5pt);
        \draw[fill=black] (5,5) circle (5pt);
        \node at (3,-1) {$m_4$};

    \end{scope}
\end{tikzpicture}

%% file: figures/interleaving_column2.tex
\begin{tikzpicture}[scale=0.7]
    \draw (0,4)--(0,0)--(3,0);
    \draw (-1,1)--(3,1);
    \draw (3,-1)--(3,4);
    \draw[dashed] (1,0)--(1,4);
    \draw[dashed] (2,0)--(2,4);
    \draw[fill=black] (0,0) circle (2pt);
    \draw[fill=black] (-1,1) circle (2pt);
    \draw[fill=black] (3,-1) circle (2pt);
    \highlightRegion{gray}{(1,0)--(1,1)--(2,1)--(2,0)}

    \foreach \x in {10, 12, 14, 16, 18, 20} {
        \draw[fill=black] (0.1*\x,2.8*0.1*\x-1.7) circle (1pt);
    }
    \draw[fill=black] (0.33, 0.2) circle (1pt);
    \draw[fill=black] (0.66, 0.4) circle (1pt);
    \draw[fill=black] (2.66, 0.6) circle (1pt);
    \draw[fill=black] (2.33, 0.8) circle (1pt);
\end{tikzpicture}

%% file: figures/decomp_UmDR.tex
\begin{tikzpicture}[yscale=-1, scale=0.5]
    \draw[black, fill=black] (4.5, 0.5) circle (3pt);
    \highlightRegion{gray!50}{(2,0)--(2,2)--(3,2)--(3,3)--(4,3)--(4,0)}
    \highlightRegion{gray!50}{(5,0)--(8,0)--(8,3)--(5,3)}
    \highlightRegion{pink}{(5,3)--(5,5)--(6,5)--(6,6)--(7,6)--(7,7)--(8,7)--(8,3)}

    \node at (0,3) (s) {\scriptsize$x(s+1)$};
    \draw[->] (s)--(1.5,0.5);
    \node at (7,-1) (y1) {\scriptsize$x(y+1)$};
    \draw[->] (y1)--(4.5,1.5);
    \draw[->] (y1)--(4.5,2.5);
    \node at (2,-1) (y) {\scriptsize$y$};
    \draw[->] (y)--(4.25,0.25);
    \node at (2,5) (z) {\scriptsize$x(z+1)$};
    \draw[->] (z)--(4.5,3.5);
    \node[right] at (9,6) (H) {$\frac{\gfUmDR(x,y,z,s)-1}{x(s+1)}$};
    \draw[->] (H)--(7,5);
    \drawGrid{7}
\end{tikzpicture}

%% file: 11-wilf.tex
\section{Unbalanced Wilf-equivalence}\label{sec:wilf}

A combination of many of our results can be used to prove that some classes
have the same enumeration.
When two classes have the same enumeration, we say that they are
\emph{Wilf-equivalent}. This equivalence is said to be \emph{unbalanced} if
one of the bases has a pattern of size $k$, but the other basis does not, like
in the next theorem.

\begin{theorem}\label{thm:wilf-same-core}
    The classes $\Av(2413, 2134,1234)$ and $\Av(2413, 2134, 1324,12534)$ are
    Wilf-equivalent.
\end{theorem}
\begin{proof}
    The first class in our notation is $\Av(r_d, 2134, 1234)$. Hence, by
    Corollary~\ref{cor:rd_2134}, its generating function $A_1(x)$ is
    \[ \gfDmUR\left(x, \frac{x}{1-x}, \frac{B_1(x)-(1+x)}{x}, C_1(x)-1,
        B_1(x)-1\right)\]
    where
    \begin{itemize}
        \item $B_1(x)$ is the generating function of $\Av(r_d, 2134, 123) = \Av(r_d,123)$.
        \item $C_1(x)$ is the generating function of $\Av(213, 12) = \Av(12)$.
    \end{itemize}

    The second class is $\Av(r_d, 2134, 1324,12534)$. Again, by
    Corollary~\ref{cor:rd_2134}, its generating function $A_2(x)$ is
    \[ \gfDmUR\left(x, \frac{x}{1-x}, \frac{B_2(x)-(1+x)}{x}, C_2(x)-1,
        B_2(x)-1\right)\]
    where
    \begin{itemize}
        \item $B_2(x)$ is the generating function of $\Av(r_d, 2134, 213, 1423) = \Av(213,1423)$.
        \item $C_2(x)$ is the generating function of $\Av(213, 21, 1423) = \Av(21)$.
    \end{itemize}

    The class $\Av(r_d,123)$ is the same class as $\Av(r_u,c_u,r_d,123)$ since
    $123$ is contained in $r_u$ and $c_u$.
    Moreover, the last class is a symmetry of $\Av(r_u, c_u, c_d, 123)$.
    Hence, by Corollary~\ref{cor:gf_rucupi},
    \[B_1(x)= \gfUDC\left(x, \frac{x}{1-x}, 1\right).\]
    We also have that $\Av(213,1423)$ is a symmetry of $\Av(132, c_u)$ which is
    the same class as $\Av(r_d, c_d, c_u, 132)$.
    Hence, by Corollary~\ref{cor:gf_rdcdpi},
    \[B_2(x) = \gfUDC\left(x, \frac{x}{1-x},
        \frac{\frac{x}{1-x}}{\frac{x}{1-x}}\right)
        = \gfUDC\left(x, \frac{x}{1-x}, 1\right).\]

    We showed that $B_1(x)=B_2(x)$ and we know that
    $C_1(x)=C_2(x)=\frac{1}{1-x}$.
    Therefore, we have that $A_1(x)=A_2(x)$.
\end{proof}

\begin{theorem}\label{thm:wilf-diff-core}
    The classes $\Av(2134, 2413)$ and $\Av(2314, 3124, 13524, 12435)$ are Wilf-equivalent.
\end{theorem}
\begin{proof}
    In our notation, the two classes are $\Av(r_d,2134)$ and
    $\Av(r_u, c_u, 13524, 12435)$.
    Let $A_1(x)$ be the generating
    function of $\Av(2134, r_d)$ as computed in Section~\ref{sec:rd_2134}.
    Let $A_2(x)$ be the generating function of $\Av(r_u, c_u, 13524,12435)$.
    By Corollary~\ref{cor:gf_upcore},
    $A_2(x)=\gfU(x,  B(x)-1)$ where  $B(x)$  is  the  generating
    function  of  $\Av(r_u,  c_u,  r_d,  1324)$.
    As this class is a symmetry of $\Av(r_u, c_u, c_d, 1324)$,
    by Corollary~\ref{cor:gf_rucupi}, $B(x)=\gfUDC(x, C(x)-1)$ where $C(x)$ is the
    generating function for $\Av(r_u, c_u, r_d, 213)$. The previous class
    is in fact $\Av(213)$. Hence, $C(x)$ is the generating function for the Catalan
    numbers. Rewinding the previous step, we can compute $A_2(x)$ explicitly.
    A simple verification then shows that $A_1(x)=A_2(x)$.
\end{proof}

%% file: 12-conclusion.tex
\section{Conclusion}\label{sec:conclusion}

The technique of enumeration studied in this paper has been implemented in
the python package \emph{Permuta}, developed by the \cite{permuta}.
To test if any of the theorems in this paper
apply to a basis one can use the code snippet below.
It will print a reference to any of the results in the paper that apply to
any symmetry of the basis of interest.

\lstset{
  backgroundcolor=\color{white},   
  basicstyle=\footnotesize\ttfamily,        
  breakatwhitespace=false,         
  breaklines=true,                 
  captionpos=b,                    
  commentstyle=\color{green},    
  keepspaces=true,                 
  keywordstyle=\color{blue},       
  rulecolor=\color{black},         
}
\begin{lstlisting}[language=Python]
from permuta import Perm
from permuta.enumeration_strategies import find_strategies

basis = [Perm((1,2,0,3)), Perm((2,0,1,3)), Perm((0,1,2,3))]
for strat in find_strategies(basis):
    print(strat.reference())
\end{lstlisting}

\paragraph*{Tracking more bases with the same cores}

Let $\pi$ be a skew-indecomposable permutation.
In Theorem~\ref{thm:inf_upcore}, we described the structure of bases of the form
$\{r_u, c_u, 1\oplus\pi\}$.
It seems possible to enumerate classes whose basis are of the form
$\{r_u, c_u, 21\oplus\pi\}$. In this
case, the staircase encoding would contain permutations avoiding
$\{r_u, c_u, \pi\}$ in the cells that are not in the leading diagonal and
permutations avoiding $\{r_u, c_u, 21\oplus\pi\}$ in the cells
in the leading diagonal.
Hence, tracking the vertices of the independent set that are in the leading
diagonal would be sufficient to enumerate this class.
It is likely that this reasoning can be extended to replace $21$
with an arbitrary decreasing sequence.

\paragraph*{Increasing the size of the patterns.}
In Section~\ref{sec:3to4}, we go from size $3$ to size $4$ patterns.
To do so, we gave a set of patterns of size $4$ that
put the same constraints on the staircase grid as $123$ did. This idea is not
limited to size $4$ patterns. We can easily see that
the nine patterns of size $5$ in Figure~\ref{fig:size5} enforce the same
constraints as $123$ on the staircase grid. Therefore, one can expect the
results to generalize to greater size. However, to do so,
two major issues need to be overcome. First, one needs to be sure that if
a pattern occurs in a permutation, then there needs to be an occurrence of the
pattern using the left-to-right minima of the permutation.
This can be done with the addition of other patterns to the basis.
Computation shows that for the size $5$ cases, $7$ patterns is the
smallest number of patterns that can be added to do so. Secondly,
the technique will not give information about the permutations with exactly
two left-to-right minima.

\begin{figure}[htpb]
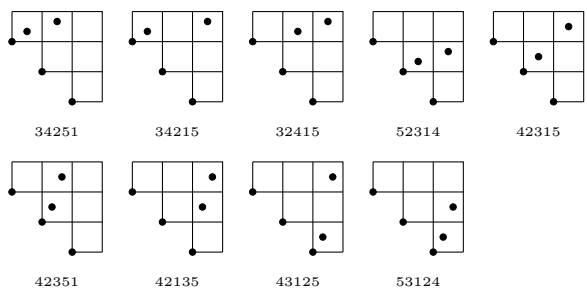

  \centering
  \includestandalone{figures/size5}
  \caption{Patterns of size $5$ that enforce the up-core constraint.}%
  \label{fig:size5}
\end{figure}

\paragraph*{Independent sets on boundary grids.}
The letter $\pi_i$ in a permutation $\pi$ is a right-to-left maximum if
$\pi_j<\pi_i$ of all $j>i$. Building a skew-shaped grid from the
left-to-right minima and the right-to-left maxima we get the \emph{boundary
  encoding} of the permutation. Figure~\ref{fig:boundary_encoding} shows an
example of a boundary encoding.
One might be able to use the boundary encoding to
generalize the method of the staircase encoding, but in this new case, any
permutation avoiding $123$ could potentially be a boundary.
\begin{figure}[htpb]
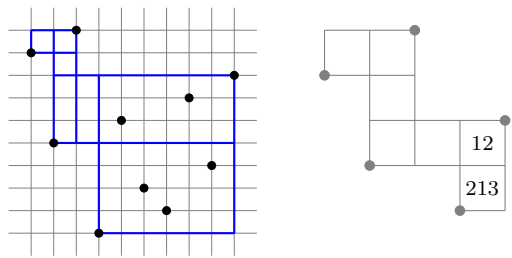

  \centering
  \includestandalone{figures/bound_enc}
  \caption{Boundary encoding of the permutation $95(10)1632748$.}%
  \label{fig:boundary_encoding}
\end{figure}

\paragraph*{Wilf-equivalence and bijective proof.}
In Section~\ref{sec:wilf}, we uncovered two Wilf-equivalences by computing the
generating functions with our results. The proof of
Theorem~\ref{thm:wilf-same-core} nicely highlights a structural argument for the
Wilf-equivalence as both classes are built from the same core. However, it is
not the case in the proof of Theorem~\ref{thm:wilf-diff-core}.
It would be interesting to establish a bijection between the two classes using
the core structure.

\paragraph*{Random sampling}

Most of the techniques used throughout can be directly, or with minor
changes, translated to the language of combinatorial specifications
(see \cite{flajolet2009analytic}).
This opens the door to many existing tools, including, but certainly not
limited to generating the permutations in a class, or uniformly sampling
permutations in a class (see \cite{MR1290534, MR2095975}).

For example, it is not too hard to see how to convert
the argument used in \cite{bean2015pattern} to find $\gfU(x,y)$, the generating
function for independent sets of the down-cores, to a specification. From
Corollary~\ref{cor:gf_downcore}, we know that $A(x)$, the generating
function of $\Av(r_d,c_d)$, satisfies $A(x)=\gfU(x, A(x)-1)$. If we work with
the $q,t$-analog of $A(x)$ where $q$ tracks
the number of left-to-right minima of the permutations
and $t$ tracks the number of active cells, that is
\[
  A(x,q,t) = \gfU(qx,tA(x)-t),
\]
then by looking at this distribution we can pick randomly the number of minima,
and active cells. We can then use the specification for $\gfU(x,y)$,
and standard techniques to uniformly sample an independent set from the
down-core of appropriate size. Finally, by recursively choosing the
permutations to fill the corresponding active cells in the same manner, we will
be able to sample uniformly from $\Av(r_d,c_d)$.

This approach could be applied to the methods within our paper, perhaps
requiring some extra `bookkeeping' along the way.

%% file: figures/size5.tex
\def\drawGrid{
    \foreach \x in {1,2,3} {
        \draw[] (\x,0)--(\x,\x)--(4,\x);
        \draw[fill=black] (\x,\x) circle (3pt);
    }
    \draw[] (1,0)--(4,0)--(4,3);
}

\newcommand\drawLabel[1]{\node at (2.5, 4) {\tiny$#1$};}

\begin{tikzpicture}[yscale=-1, scale=0.4]
    \drawGrid
    \draw[fill=black] (1.5, 0.66) circle (3pt);
    \draw[fill=black] (2.5, 0.33) circle (3pt);
    \drawLabel{34251}
    \begin{scope}[shift={(4,0)}]
        \drawGrid
        \draw[fill=black] (1.5, 0.66) circle (3pt);
        \draw[fill=black] (3.5, 0.33) circle (3pt);
        \drawLabel{34215}
    \end{scope}
    \begin{scope}[shift={(8,0)}]
        \drawGrid
        \draw[fill=black] (2.5, 0.66) circle (3pt);
        \draw[fill=black] (3.5, 0.33) circle (3pt);
        \drawLabel{32415}
    \end{scope}
    \begin{scope}[shift={(12,0)}]
        \drawGrid
        \draw[fill=black] (2.5, 1.66) circle (3pt);
        \draw[fill=black] (3.5, 1.33) circle (3pt);
        \drawLabel{52314}
    \end{scope}

    \begin{scope}[shift={(0,5)}]
        \drawGrid
        \draw[fill=black] (2.33, 1.5) circle (3pt);
        \draw[fill=black] (2.66, 0.5) circle (3pt);
        \drawLabel{42351}
    \end{scope}
    \begin{scope}[shift={(4,5)}]
        \drawGrid
        \draw[fill=black] (3.33, 1.5) circle (3pt);
        \draw[fill=black] (3.66, 0.5) circle (3pt);
        \drawLabel{42135}
    \end{scope}
    \begin{scope}[shift={(8,5)}]
        \drawGrid
        \draw[fill=black] (3.33, 2.5) circle (3pt);
        \draw[fill=black] (3.66, 0.5) circle (3pt);
        \drawLabel{43125}
    \end{scope}
    \begin{scope}[shift={(12,5)}]
        \drawGrid
        \draw[fill=black] (3.33, 2.5) circle (3pt);
        \draw[fill=black] (3.66, 1.5) circle (3pt);
        \drawLabel{53124}
    \end{scope}
    \begin{scope}[shift={(16,0)}]
        \drawGrid
        \draw[fill=black] (3.5, 0.5) circle (3pt);
        \draw[fill=black] (2.5, 1.5) circle (3pt);
        \drawLabel{42315}
    \end{scope}
\end{tikzpicture}

%% file: figures/bound_enc.tex
\begin{tikzpicture}[scale=0.3]
    \foreach \x in {1,...,10}{%
        \draw[ultra thin, gray] (\x,0)--(\x,11); 
        \draw[ultra thin, gray] (0,\x)--(11,\x); 
    }
    \draw[thick, blue] (1,10)--(1,9)--(3,9);
    \draw[thick, blue] (2,10)--(2,5)--(10,5);
    \draw[thick, blue] (1,10)--(3,10)--(3,5);
    \draw[thick, blue] (2,8)--(10,8)--(10,1);
    \draw[thick, blue] (4,8)--(4,1)--(10,1);

    \draw[fill=black] (1,9) circle (5pt);
    \draw[fill=black] (2,5) circle (5pt);
    \draw[fill=black] (3,10) circle (5pt);
    \draw[fill=black] (4,1) circle (5pt);
    \draw[fill=black] (5,6) circle (5pt);
    \draw[fill=black] (6,3) circle (5pt);
    \draw[fill=black] (7,2) circle (5pt);
    \draw[fill=black] (8,7) circle (5pt);
    \draw[fill=black] (9,4) circle (5pt);
    \draw[fill=black] (10,8) circle (5pt);

    \begin{scope}[shift={(12,0)}, scale=2]
        \draw[gray] (1,5)--(1,4)--(3,4);
        \draw[gray] (2,5)--(2,2)--(5,2);
        \draw[gray] (1,5)--(3,5)--(3,2);
        \draw[gray] (2,3)--(5,3)--(5,1);
        \draw[gray] (4,3)--(4,1)--(5,1);
        \draw[gray, fill=gray] (1,4) circle (3pt);
        \draw[gray, fill=gray] (2,2) circle (3pt);
        \draw[gray, fill=gray] (3,5) circle (3pt);
        \draw[gray, fill=gray] (5,3) circle (3pt);
        \draw[gray, fill=gray] (4,1) circle (3pt);
        \node at (4.5, 1.5) {\footnotesize$213$};
        \node at (4.5, 2.5) {\footnotesize$12$};
    \end{scope}
\end{tikzpicture}